\documentclass[a4paper,11pt]{article}
\usepackage{amsmath}
\usepackage{amsfonts}
\usepackage{titlesec}
\usepackage{amssymb}
\usepackage{amsthm}
\usepackage{enumitem}
\usepackage{mathtools}
\usepackage{listings}
\usepackage{geometry}
\usepackage{color}

\definecolor{mycolor1}{rgb}{0.00000,0.44700,0.74100}%

\usepackage[utf8]{inputenc}

\newcommand{\quotes}[1]{``#1''}

\usepackage{array}
\makeatletter
\newcommand{\thickhline}{%
    \noalign {\ifnum 0=`}\fi \hrule height 1pt
    \futurelet \reserved@a \@xhline
}
\newcolumntype{"}{@{\hskip\tabcolsep\vrule width 1pt\hskip\tabcolsep}}
\makeatother

\usepackage{framed}
\usepackage{graphicx, caption, subcaption}
\usepackage{hyperref}
\usepackage{dsfont}
\usepackage{float}

\usepackage{tikz,pgfplots}

\newtheorem{theorem}{Theorem}[section]

\newtheorem{proposition}{Proposition}[section]
\newtheorem{lemma}[theorem]{Lemma}
\newtheorem{remark}[theorem]{Remark}
\newtheorem{definition}[theorem]{Definition}
\newtheorem{conclusion}[theorem]{Conclusion}

\newcommand{\dx}{\hspace{2pt} \mbox{d}x}
\newcommand{\eps}{\varepsilon}

\newcommand{\R}{\mathbb{R}}

\usepackage{diagbox}

\renewcommand{\thefootnote}{\fnsymbol{footnote}}
\renewcommand{\thefootnote}{\arabic{footnote}}

\begin{document}

\begin{center}
{\LARGE The dependency of spectral gaps on the convergence of the inverse iteration for a nonlinear eigenvector problem
\renewcommand{\thefootnote}{\fnsymbol{footnote}}\setcounter{footnote}{0}
 \hspace{-3pt}\footnote{The author acknowledges the support by the G\"oran Gustafsson foundation and the German Research Foundation (DFG grant HE 2464/7-1).}}\\[2em]
\end{center}

\begin{center}
{\large Patrick Henning\footnote[1]{Department of Mathematics, Ruhr-University Bochum, DE-44801 Bochum, Germany
}}\\[2em]
\end{center}

\begin{center}
{\large{\today}}
\end{center}

\begin{center}
\end{center}

\begin{abstract}
In this paper we consider the generalized inverse iteration for computing ground states of the Gross--Pitaevskii eigenvector problem (GPE). For that we prove explicit linear convergence rates that depend on the maximum eigenvalue in magnitude of a weighted linear eigenvalue problem. Furthermore, we show that this eigenvalue can be bounded by the first spectral gap of a linearized Gross-Pitaevskii operator, recovering the same rates as for linear eigenvector problems. With this we establish the first local convergence result for the basic inverse iteration for the GPE without damping. We also show how our findings directly generalize to extended inverse iterations, such as the Gradient Flow Discrete Normalized (GFDN) proposed in [{\it W. Bao, Q. Du, SIAM J. Sci. Comput., 25 (2004)}] or the damped inverse iteration suggested in  [{\it P. Henning, D. Peterseim, SIAM J. Numer. Anal., 53 (2020)}]. Our analysis also reveals why the inverse iteration for the GPE does not react favourably to spectral shifts. This empirical observation can now be explained with a blow-up of a weighting function that crucially contributes to the convergence rates. Our findings are illustrated by numerical experiments.
\end{abstract}

\section{Introduction}
When a dilute gas of bosons is cooled down to ultra-low temperatures approaching $0K$, a so-called Bose--Einstein condensate (BEC) is formed \cite{PiS03}. Such condensates are an extraordinary state of matter, where a collective of particles behaves as if they were one single \quotes{super atom} that becomes macroscopically visible and which allows to study quantum phenomena on an observable scale. Of particular interest are the ground sates of a Bose--Einstein condensate, i.e., the lowest energy states of a BEC that is trapped in a magnetic potential. The most popular mathematical model for describing such ground states of BECs is given by the Gross--Pitaevskii eigenvector problem (GPE). The equation is named after E. Gross and L. Pitaevskii who first derived it \cite{Gro61,Pit61}. The GPE seeks a normalized eigenfunction $u$ for the smallest eigenvalue $\lambda\in \R$ such that
\begin{align}
\label{initial-GPE}
- \tfrac{1}{2} \Delta u + V\, u + \beta |u|^2 u = \lambda \, u 
\end{align}
in some domain $\Omega \subset \mathbb{R}^d$ in either $1d$, $2d$ or $3d$. Here, the eigenfunction $u$ describes the quantum state of the BEC, $|u|^2$ is its density and the eigenvalue $\lambda$ can be interpreted as the chemical potential. The trapping potential is model by the function $V(x)$ and the constant $\beta \in \mathbb{R}$ describes the strength and the direction of particle interactions. In this paper we will only consider the regime of repulsive interactions, which means that $\beta\ge0$. It is worth to note that in the real-valued setting of this work, it does not matter if we write $|u|^2u$ or $u^3$ in equation \eqref{initial-GPE}. However, in more general complex valued settings (e.g. if there is a time-dependency) this is no longer the case. For a consistent notation throughout the mathematical and physical literature, the GPE is hence always written with $|u|^2u$ in the nonlinear term as this expression still remains correct if $u$ is complex-valued.

The mathematical literature on the GPE and its numerical treatment is vast and we refer to the survey articles by Bao et al. \cite{Bao14,BaC13b} for an illustrating introduction to the topic and for a detailed literature overview. 

In the following discussion we will mainly focus on iterative methods for solving the eigenvalue problem \eqref{initial-GPE} in the Hilbert space $H^1_0(\Omega)$. Practically however, a numerical method also requires a suitable space discretization of finite element-, finite difference-, or spectral-type. The approximation properties of such discrete ground states with respect to an exact ground state were analyzed in detail in \cite{CCM10}. Furthermore, the usage of generalized finite element spaces with significantly improved approximation properties was proposed in \cite{HMP14b,HeP21}. Other techniques that improve the quality of discrete minimizers are adaptivity \cite{HSW21,XiX19} and two-grid post-processing \cite{CCH18,HMP14b}. 

With this remark, we return to the discussion of iterative solvers and we neglect from now on the aspect of the space discretization. Broadly speaking, there are three major classes of iterative methods to approach the GPE \eqref{initial-GPE}: self-consistent field iterations (SCF), generalized inverse iterations (\quotes{A-methods}) and inverse iterations based on a modified scaling-invariant operator (\quotes{J-methods}). The self-consistent field iterations are based on linearizing the GPE with a given approximation $u^n$ and to afterwards solve the arising {\it linear} eigenvalue problem. The result from the linear eigenvalue problem is then used as an updated approximation $u^{n+1}$ and the procedure is repeated. Methods that fall into this class are e.g. discussed and analyzed in \cite{CKL21,CaL00,CaL00B,DiC07,UJR21}. 
The second class, i.e., generalized inverse iterations, are also based on linearizing the GPE, however instead of solving an eigenvalue problem, the inverse of the linearized operator is simply applied to a previous approximation and the result is normalized afterwards. Such inverse iterations can be interpreted and derived in the context of projected Sobolev-gradient flows \cite{BaD04,BWM05,DaK10,HeP20,KaE10,LiC21,Zha21}, Riemannian optimization \cite{APS21,DaP17} or preconditioned (conjugate) gradient methods \cite{ALT17,AnD14}. An approach that does not fit in one of the aforementioned classes is known as J-method and is based on modifying the Gross--Pitaevskii operator so that it becomes scaling-invariant. The derivative of this modified operator (the so called J-operator) is then used in an inverse iteration with shift to accelerate the convergence. The method was first proposed in \cite{JarKM14} and further analyzed in \cite{AHP21NumMath}.

As indicated by the discussion, the generalized inverse iterations (A-methods) constitute by far the largest class of approaches for the GPE. Despite its popularity and long history to solve the equation, there are still several fundamental questions regarding its convergence that are open. The first proof of global convergence for an inverse iteration with adaptive damping was presented in \cite{HeP20}, however without quantifying the speed of convergence. An improvement was recently obtained by Zhang \cite{Zha21} who proved that the convergence must be linear for all sufficiently small damping parameters, i.e., there exists a contraction rate $r<1$, such that $\| u^{n} - u \|_{H^1(\Omega)} \le C \, r^n \, \| u^{0} - u \|_{H^1(\Omega)}$, where $u$ is the exact ground state and $u^n$ are the approximations obtained from the inverse iterations with damping for starting value $u^0$. Shortly after, the techniques from \cite{HeP20,Zha21} were further generalized to other equations, such as the Kohn--Sham model \cite{APS21}. However, two central questions that remained open are:
\begin{itemize}
\item Can we still guarantee convergence when the damping parameter is removed (i.e. for the basic inverse iteration)? 
\item Even more importantly, can we quantify the linear convergence rate $0<r<1$ depending on spectral gaps, analogously to the linear case?
\end{itemize} 
In this paper we will answer both questions positively and we will transfer our results to damped versions of the inverse iteration method, such as the GFDN (Gradient Flow Discrete Normalized) by Bao and Du  \cite{BaD04} and the discrete Sobolev gradient flow proposed in \cite{HeP20}. With this, our findings also complement the local convergence result for the GFDN obtained by Faou and J\'{e}z\'{e}quel \cite{FaJ18} in $1d$ for focussing nonlinearities (i.e. in the regime $\beta <0$).
Furthermore, our analysis reveals why introducing a shift $\sigma$ in the direction of the ground state eigenvalue $\lambda$ will typically not improve the convergence, but might in fact ruin it. This numerically observed phenomenon (cf. \cite{AHP21NumMath,JarKM14}) is related to a pollution factor of order $\beta|\lambda-\sigma|^{-1}$ that influences the convergence rates.\\[0.3em]
{\bf Outline.} The paper is structured as follows. In Section \ref{section-setting} we give a precise mathematical description of the Gross--Pitaevskii equation and we recall important results concerning the existence and uniqueness of ground states, as well as equivalent characterizations. In Section \ref{section-basic-inverse-iteration} we formulate the basic version of the inverse iteration for the GPE and we present and prove our first main result concerning explicit asymptotic convergence rates. A generalization of our findings to the GFDN method are given in Section \ref{section-GFDN}. Based on these findings, we investigate a hypothetical inverse iteration with shift in Section \ref{section-inverse-iteration-with-shift}, where we can give mathematical arguments why a good performance of such a method cannot be expected in general. The inverse iteration with a damping parameter is analyzed in Section \ref{section-inverse-iteration-with-damping}. Finally, we conclude with numerical experiments in Section \ref{section-numerical-experiments}.

\section{Analytical setting and preliminaries}
\label{section-setting}

In the following, we let 
\begin{enumerate}[label={(A\arabic*)}]
\item \label{A1}$\Omega \in \mathbb{R}^d$ be a bounded and convex domain in dimension $d=1,2,3$.
\end{enumerate}
On $\Omega$, we use standard notation for Lebesgue and Sobolev spaces. As a simplifying notation, the $L^2$-norm shall be denoted by $\| \cdot \| := \| \cdot \|_{L^2(\Omega)}$. Furthermore, the Sobolev space of $L^2$-integrable and weakly-differentiable functions with a vanishing trace on the boundary $\partial \Omega$ is as usual denoted by $H^1_0(\Omega)$. 
The dual space of $H^1_0(\Omega)$ is given by $H^{-1}(\Omega)$.\\[0.1em]
Given 
\begin{enumerate}[resume,label={(A\arabic*)}]
\item\label{A2} a potential $V\in L^{\infty}(\Omega)$ with $V(x)\ge 0$ for almost all $x\in \Omega$ , and
\item\label{A3} a real-valued repulsion constant $\beta\ge 0$,
\end{enumerate}
we define the Gross--Pitaevskii energy functional $E: H^1_0(\Omega) \to \R$ by
\begin{align}
\label{definition-energy-functional}
E(v) := \frac{1}{2}\int_{\Omega} \tfrac{1}{2}|\nabla v|^2 + V |v|^2 + \tfrac{\beta}{2} |v|^4 \dx.
\end{align}
With this, we seek a minimizer $u$ of $E$ under the mass-normalization constraint that $\| u\| =1$ (which represents the conservation of a certain particle number). Such a minimizer is called a {\it ground state} and describes, for example, the lowest energy states of Bose--Einstein condensates. Since the functional $E$ is weakly lower semi-continuous on $H^1_0(\Omega)$ and bounded from below by zero, the existence of a minimizer is apparent. Similarly, with the diamagnetic inequality $|\nabla|v|| \le |\nabla v|$ for all $v\in H^1_0(\Omega)$ it follows that if $u$ is a ground state, then $|u| \in H^1_0(\Omega)$ is also a ground state. With this insight, it is possible to prove that the minimizer $u$ must be either strictly positive or strictly negative in $\Omega$ and that it is  unique up to its sign, cf. \cite[Appendix]{CCM10}. From the latter reference, the following result can be extracted:
\begin{proposition}[Minimizers of the Gross--Pitaevskii energy]
\label{proposition-minimization-problem}
Assume \ref{A1}-\ref{A3}, then there exist exactly two $L^2$-normalized minimizers $u$ of $E$ with
\begin{align}
\label{def-ground-state}
u = \mbox{\rm arg\hspace{2pt}min} \{ E(v) | \hspace{2pt} v\in H^1_0(\Omega), \, \| v\| =1 \},
\end{align} 
which are $|u|$ and $-|u|$. Furthermore, it holds $u \in H^2(\Omega) \cap C^{0,\alpha}(\overline{\Omega})$ for some $0<\alpha < 1$.
\end{proposition}
Without loss of generality, we consider from now on only the (unique) positive ground state $u$.

Any minimizer of the constrained minimization problem \eqref{def-ground-state} can be equivalently expressed by the Euler-Lagrange equations seeking $u \in H^1_0(\Omega)$ and minimal $\lambda>0$ such that
\begin{align}
\label{GPEV-preliminary}
\langle E'(u) , v \rangle = \lambda \, ( u , v)_{L^2(\Omega)} \qquad \mbox{for all } v\in H^1_0(\Omega).
\end{align}
Here $E^{\prime}$ denotes the Fr\'echet-derivative of $E$, $\langle \cdot ,\cdot \rangle$ denotes the canonical duality pairing on $H^1_0(\Omega)$ and $H^{-1}(\Omega)$ and $\lambda$ is the Lagrange multiplier of the constraint $\| u \| =1$. The Fr\'echet derivative can be be computed as
\begin{align}
\label{Eprime}
\langle E'(v) , w \rangle = \tfrac{1}{2}(\nabla v ,\nabla w)_{L^2(\Omega)} + (V\,  v , w)_{L^2(\Omega)} 
+ \beta (|v|^2 v ,w )_{L^2(\Omega)}. 
\end{align}
Hence, problem \eqref{GPEV-preliminary} can be seen as an eigenvalue problem with eigenfunction nonlinearity, which will be also the viewpoint that we take in this paper. Also note that it is not obvious that minimizing the eigenvalue $\lambda$ leads to a ground state in the sense of  \eqref{def-ground-state}. A corresponding proof of this observation can be found in \cite{CCM10}. In particular, we have the following proposition.
\begin{proposition}[Gross--Pitaevskii eigenvector problem]
\label{proposition-ev-problem}
Assume \ref{A1}-\ref{A3}, then the unique positive ground state $u$ to \eqref{def-ground-state} can be equivalently expressed as seeking the smallest eigenvalue $\lambda>0$ and the corresponding eigenfunction $u\in H^1_0(\Omega)$ with $\| u \|=1$ and $u>0$ in $\Omega$ such that
\begin{align}
\label{groundstate-ev-problem}
a_u(u , v ) =  \lambda \, ( u , v)_{L^2(\Omega)} \qquad \mbox{for all } v\in H^1_0(\Omega).
\end{align}
For fixed $z\in H^1_0(\Omega)$, the coercive and continuous bilinear form $a_z(\cdot,\cdot) : H^1_0(\Omega) \times H^1_0(\Omega) \rightarrow \R$ is defined by
\begin{align*}
a_z(w,v):= \tfrac{1}{2}(\nabla w ,\nabla v)_{L^2(\Omega)} + (V\,  w , v)_{L^2(\Omega)}  + \beta (|z|^2 w ,v)_{L^2(\Omega)}.
\end{align*}
The eigenvalue $\lambda$ is simple.
\end{proposition}
Propositions \ref{proposition-minimization-problem} and \ref{proposition-ev-problem} give a comprehensive characterisation of ground states. Before we proceed with how to compute ground state approximations with a generalized inverse iteration, we require a linear auxiliary problem that is crucial for specifying explicit convergence rates. For that we seek eigenvalues $\lambda_i \in \R$ and corresponding eigenfunctions $u_i \in H^1_0(\Omega)$ with $\| u_i \|=1$ such that
\begin{align}
\label{linearized-ev-problem}
a_u ( u_i , v ) = \lambda_i \, ( u_i , v )_{L^2(\Omega)} \qquad \mbox{for all } v \in H^1_0(\Omega). 
\end{align}
Here, $u\in H^1_0(\Omega)$ denotes the unique positive ground state given by \eqref{def-ground-state}. Since $a_u(\cdot,\cdot)$ is a coercive, continuous and symmetric bilinear form, standard results of spectral theory guarantee that there exists a countably infinite number of eigenvalues
$$
0< \lambda_1 < \lambda_2 \le \lambda_3 \le ... \rightarrow \infty.
$$ 
The positivity follows from the coercivity of $a_{u}(\cdot,\cdot)$ and the observation that $\lambda_1$ is simple follows again from the fact that $u_1$ cannot be sign changing. As $u$ is obviously an eigenfunction of \eqref{linearized-ev-problem} that is strictly positive and since there cannot be a second eigenfunction with that property (due to $L^2$-orthogonality of eigenfunctions), we conclude that $u_1=u$ (up to sign) and $\lambda=\lambda_1$. For more details on the argument we refer again to \cite{CCM10}. The findings are summarized in the following proposition which finishes our preparations on the analytical background. 
\begin{proposition}
\label{proposition-linearized-problem}
Assume \ref{A1}-\ref{A3} and consider the {\rm linearized} eigenvalue problem \eqref{linearized-ev-problem} with smallest eigenvalue $\lambda_1$ (which is simple) and the second smallest eigenvalue $\lambda_2$. 
If $u$ denotes the ground state eigenfunction to the nonlinear problem \eqref{groundstate-ev-problem} (respectively \eqref{def-ground-state}) with ground state eigenvalue $\lambda$, then we have
$$
\lambda_1=\lambda, \qquad |u_1|=|u| \qquad  \mbox{and} \qquad \tfrac{\lambda}{\lambda_2} < 1.
$$
\end{proposition}

\section{Basic inverse iteration and convergence rates}
\label{section-basic-inverse-iteration}

A classical method for computing the smallest eigenvalue of a linear eigenvalue problem is the inverse iteration method, which requires to apply the inverse of a differential operator repeatedly to a given starting function and normalize the result after each application of the inverse. Since the Gross-Pitaevskii eigenvector problem \eqref{groundstate-ev-problem} is nonlinear and a direct inversion of the nonlinear operator is impractical, it is reasonable to first linearize the differential operator using the previous approximation and to then invert it afterwards. This simple procedure leads to the basic version of the inverse iteration for the Gross--Pitaevskii eigenvector problem (GPE). This basic version (without damping or other modifications) is e.g. considered in the numerical experiments presented in \cite{HSW21,HeP20,JarKM14}.

\begin{definition}[Basic inverse iteration for the GPE]
\label{definition-basic-inverse-iteration}
For $v\in H^1_0(\Omega)$, let $\mathcal{A}_v : H^1_0(\Omega) \rightarrow H^{-1}(\Omega)$ denote the $v$-linearized differential operator given by
\begin{align*}
\mathcal{A}_v w := a_v(w , \cdot ) \qquad\mbox{for } w \in H^1_0(\Omega).
\end{align*}
Its inverse, interpreted as an isomorphism on $H^1_0(\Omega)$, is denoted by $\mathcal{G}_v : H^1_0(\Omega) \rightarrow H^1_0(\Omega)$, with
\begin{align*}
\mathcal{G}_{v}  w :=  \mathcal{A}_{v}^{-1}\mathcal{I} w, 
\end{align*}
where $\mathcal{I} : H^1_0(\Omega) \rightarrow H^{-1}(\Omega)$ is the canonical identification $\mathcal{I}v:=(v,\cdot)_{L^2(\Omega)}$.

Given an initial value $u^0 \in H^1_0(\Omega)$ with $\| u^0 \|=1$, the generalized inverse iterations are recursively given by
\begin{align}
\label{basic-inverse-iterations}
u^{n+1} := \frac{\mathcal{G}_{u^n} (u^{n})}{\| \mathcal{G}_{u^n} (u^{n}) \|} \qquad \mbox{for } n\in \mathbb{N}.
\end{align}
Note that the normalization step after each iteration is crucial for nonlinear problems. The eigenvalue after $n$ iterations is approximated by $\lambda^{(n)}:=a_{u^n}(u^n,u^n)$. 
\end{definition}
The above method obviously generalizes the classical inverse iterations for linear eigenvalue problems which we recover for $\beta=0$. Even in its basic version, the iterations \eqref{basic-inverse-iterations} are a popular choice for computing ground states due to its simple implementation. However, despite its popularity, a proof of convergence remained open. Theorem \ref{thm_conv_basic_inverse_iteration}, which is our first main result, closes this gap. It will be proved in Section \ref{subsection-proof-main-1} below. To formulate it, we require another technical assumption \ref{A4}, which demands that there is no open ball in $\Omega$ on which the potential $V$ fully vanishes. 
\begin{enumerate}[resume,label={(A\arabic*)}]
\item\label{A4} For all $x\in \Omega$ and all $\eps>0$ it holds
$$
\| V \|_{L^{\infty}(\Omega \cap B_{\eps}(x))} > 0,
$$
where $B_{\eps}(x)$ is the ball with radius $\eps$ around $x$.
\end{enumerate}
Assumption \ref{A4} includes for example the important class of harmonic trapping potentials $V(x)=\sum_{i=1}^d \gamma_i x_i^2$ with trapping frequencies $\gamma_i>0$. Note however that assumption \ref{A4} is always uncritical since we can simply add an arbitrary small constant $\delta>0$ to $V\ge 0$ so that \ref{A4} is fulfilled for $V(x)+\delta$. This will not change the ground state $u$ and just shifts the spectrum by $\delta$.
\begin{theorem}[Local convergence of the basic inverse iterations]
\label{thm_conv_basic_inverse_iteration}
Assume \ref{A1}-\ref{A4}, let $u \in H^1_0(\Omega)$ denote the $L^2$-normalized ground state to the Gross--Pitaevskii eigenvector problem \eqref{groundstate-ev-problem} (respectively \eqref{def-ground-state}) and let $u^n\in H^1_0(\Omega)$ denote the approximations generated by the basic inverse iteration \eqref{basic-inverse-iterations}. Then there exists a neighborhood $S$ of $u$ and a constant $C>0$ such that
\begin{align}
\label{rate-basic-inverse}
\|  u^n - u \|_{H^1(\Omega)} \le C \hspace{2pt}\, |\tfrac{\lambda_1}{\lambda_2}|^n \,\hspace{2pt} \| u^0 - u \|_{H^1(\Omega)}
\end{align}
for all starting values $u^0 \in S$
and $n\ge  1$. Here $\lambda_1= \lambda>0$ is the ground state eigenvalue and $\lambda_2>\lambda_1$ is the second eigenvalue of the linearized eigenvalue problem \eqref{linearized-ev-problem}.
\end{theorem}
Proposition \ref{proposition-linearized-problem} shows
$$
|\tfrac{\lambda_1}{\lambda_2}| < 1,
$$
hence we have a guaranteed linear convergence. Noting that the Gross--Pitaevskii eigenvector problem \eqref{groundstate-ev-problem} and the linearized eigenvalue problem \eqref{linearized-ev-problem} are identical for $\beta=0$, we recover the well-known result for linear eigenvalue problems that the convergence speed of the inverse iteration depends on the size of the first spectral gap. If we drop assumption \ref{A4}, the convergence rates in \eqref{rate-basic-inverse} change to $ |\tfrac{\lambda_1 +\delta}{\lambda_2 + \delta}|$ for any $\delta>0$ and where the neighborhood $S$ would now depend on $\delta$.
\begin{remark}[Sharper characterization of the convergence rates]
\label{remark-sharper-rates}
In the setting of Theorem \ref{thm_conv_basic_inverse_iteration}, an even sharper bound for the convergence rate can be obtained. In fact, for any $\eps>0$ there exist a neighborhood $S_{\eps}$ of $u$ and a constant $C_\eps>0$ such that
\begin{align}
\label{improved-rate}
\|  u^n - u \|_{H^1(\Omega)} \le C_\eps \hspace{2pt} ||\mu_1| + \eps|^n \hspace{2pt} \| u^0 - u \|_{H^1(\Omega)},
\end{align}
where $\mu_1$ (with $|\mu_1|<1$) is the largest eigenvalue in magnitude of the weighted {\rm linear} eigenvalue problem seeking
$$
v_i \in V_{u}^{\perp} := \{ w \in H^1_0(\Omega) \,| \hspace{2pt} (u,w)_{L^2(\Omega)}=0 \}
$$
and $\mu_i \in \mathbb{R}$, such that
\begin{align}
\label{weighted-evp}
( v_i , (\lambda- 2\beta |u|^2 ) w )_{L^2(\Omega)} = \mu_i \, a_u( v_i , w ) \qquad \mbox{for all } w\in  V_{u}^{\perp}.
\end{align}
We will later see that $|\mu_1| \le \tfrac{\lambda_1}{\lambda_2}$ (with $|\mu_1| = \tfrac{\lambda_1}{\lambda_2}$  only for $\beta=0$), so that the rate \eqref{improved-rate} is indeed sharper. This is a direct conclusion of the proof presented in Section \ref{subsection-proof-main-1}. Also the main results later formulated in Theorem \ref{thm_conv_DNGF_iteration} and Theorem \ref{thm_conv_inverse_iteration_damping} still hold with this modification. In the numerical experiments in Section \ref{section-numerical-experiments} we will see that $|\mu_1|$ resembles the asymptotic convergence rate very accurately.
\end{remark}
The rest of this section is devoted to the proof of Theorem \ref{thm_conv_basic_inverse_iteration}, which falls into four major steps. In the first step we write the iterations as a fixed point iteration of the form $u^{n+1}=\phi(u^n)$ in order to apply the Ostrowski theorem. The Ostrowski theorem states that the local convergence rate of the iteration is given by the spectral radius of $\phi^{\prime}(u)$. We are therefore first concerned with computing the Fr\'echet derivative of $\phi$.  This is done in Section \ref{subsection-ostrowski}. In the second step, we need to find a suitable characterization of the spectrum of $\phi^{\prime}(u)$ based on a linearized version of the Gross--Pitaevskii operator. This lead us to the weighted eigenvalue problem \eqref{weighted-evp}. In order to bound the corresponding maximum eigenvalue in magnitude (which coincides with the local convergence rate), the third step requires the derivation of a pointwise estimate for the weighting function $\lambda- 2\beta |u|^2$ appearing in \eqref{weighted-evp}. This estimate is established by Lemma \ref{lemma-Linfty-theta-u} in Section \ref{subsection-proof-main-1}. Finally, in the last step we bound the largest eigenvalue of the weighted problem  \eqref{weighted-evp} by using the Courant--Fischer min-max principle. Here, the $L^{\infty}$-bound of the weighting function from step 3 is entering crucially. These final steps of the proof are all presented in Section \ref{subsection-proof-main-1}.

\subsection{Ostrowski theorem and calculation of Fr\'echet derivatives}
\label{subsection-ostrowski}
As sketched above, all local convergence results in this paper are established through the Ostrowski theorem in Banach spaces (cf. \cite{Ost66} for the original theorem by Ostrowski and \cite{AHP21NumMath,Shi81} for a simple proof in the general setting).
Adapted to our purposes, the result reads as follows.
\begin{proposition}[Ostrowski theorem]
\label{prop-ostrowski}
Let $\phi\colon H^1_0(\Omega) \rightarrow H^1_0(\Omega)$ be a mapping that is Fr\'echet-differentiable at $u \in H^1_0(\Omega)$ such that the Fr\'echet-derivative $\phi^{\prime}(u)\colon H^1_0(\Omega) \rightarrow H^1_0(\Omega)$ is a bounded linear operator with spectral radius 
$
\rho := \rho( \hspace{2pt} \phi^{\prime}(u) \hspace{2pt}) < 1.
$
Then there is an open $H^1$-neighborhood $S$ of $u$, such that for all starting values $u^{0} \in S$ we have that the fixed point iterations
$$
u^{n+1} := \phi( u^n )
$$ 
converge strongly in $H^1_0(\Omega)$ to $u$, i.e., $\| u^n - u \|_{H^1(\Omega)} \rightarrow 0$ for $n\rightarrow \infty$. Furthermore, for every $\eps>0$ there exists a neighborhood $S_{\eps}$ of $u$ and a  constant $C_\eps>0$ such that
$$
\|  u^n - u \|_{H^1(\Omega)} \le C_\eps \hspace{2pt} |\rho + \eps|^n \hspace{2pt} \| u^0 - u \|_{H^1(\Omega)} \qquad \mbox{for all } u^0 \in S_{\eps} \mbox{ and } n\ge  1.
$$
Hence, $\rho$ defines the asymptotic linear convergence rate of the fixed point iteration. 
\end{proposition}
As we want to apply Proposition \ref{prop-ostrowski} to prove the local convergence in Theorem \ref{thm_conv_basic_inverse_iteration}, we need to compute the Fr\'echet derivative of $\phi(v)= \tfrac{\mathcal{G}_{v} v}{\| \mathcal{G}_{v} v \|}$, evaluate it for the ground state $u$ and then estimate the spectral radius. In the first step, we shall therefore compute the Fr\'echet derivatives of $\mathcal{G}_{v} v$ and $\phi(v)$.

\begin{lemma}
\label{lemma-1}
Let $\psi : H^1_0(\Omega) \rightarrow H^1_0(\Omega)$ be the nonlinear map given by $\psi(v):=\mathcal{G}_{v} v$, then $\psi$ is Fr\'echet differentiable for all $v\in H^1_0(\Omega)$ with derivative $\psi^{\prime}(v): H^1_0(\Omega) \rightarrow H^1_0(\Omega)$ given by
\begin{align}
\label{psi-derivative}
 \psi^{\prime}(v) h = \mathcal{G}_{v} ( \hspace{1pt} (1 - 2 \beta \, v \, \mathcal{G}_{v} v) h \hspace{1pt} )
\end{align}  
in the direction $h\in H^1_0(\Omega)$. Furthermore, the map $\phi : H^1_0(\Omega) \rightarrow H^1_0(\Omega)$ with
$$
\phi(v):= \frac{\psi(v)}{\| \psi(v) \|}
$$
is also Fr\'echet differentiable for all $v\in H^1_0(\Omega) \setminus \{ 0\}$, where the corresponding derivative, $\phi^{\prime}(v): H^1_0(\Omega) \rightarrow H^1_0(\Omega)$, is given by
\begin{align*}
 \phi^{\prime}(v) h 
&= \frac{1}{ \| \mathcal{G}_{v} v \| } \mathcal{G}_{v} ( \hspace{1pt} (1 - 2 \beta \, v \, \mathcal{G}_{v} v) h \hspace{1pt} ) - \frac{1}{ \| \mathcal{G}_{v} v \|^3 }
( \, \mathcal{G}_{v} ( \hspace{1pt} (1 - 2 \beta \, v \, \mathcal{G}_{v} v) h \hspace{1pt} ) \, ,\mathcal{G}_{v} v \, )_{L^2(\Omega)}\, \mathcal{G}_{v} v
\end{align*}
for $h\in H^1_0(\Omega)$.
\end{lemma}

\begin{proof}
From the definition of $\mathcal{G}_{v}v=\psi(v)$ we have
\begin{align*}
\mathcal{A}_{v} \psi(v) = (v ,\cdot)_{L^2(\Omega)} = \mathcal{I} v,
\end{align*}
which maps $v\in H^1_0(\Omega)$ to an element of $H^{-1}(\Omega)$. For a location $v\in H^1_0(\Omega)$ and a direction $h\in H^1_0(\Omega)$ we therefore obtain for the Fr\'echet derivative
\begin{align*}
\mathcal{I}h &= [ \tfrac{\mbox{\scriptsize d}}{\mbox{\scriptsize d} v } [ \mathcal{A}_{v} \psi(v) ] (v) ] h =  \tfrac{\mbox{\scriptsize d}}{\mbox{\scriptsize d} v } [ a_v( \psi(v) , \cdot ) ](h) = 
2 \beta \, \mathcal{I} ( v \, \psi(v) \, h )
+ a_v ( \, \psi^{\prime}(v) h \, , \cdot ).
\end{align*}
Rearranging the terms, we have that $ \psi^{\prime}(v) h \in H^1_0(\Omega)$ fulfills 
\begin{align*}
a_v ( \, \psi^{\prime}(v) h \, , \cdot ) = \mathcal{I}(h - 2 \beta \, v \, \psi(v) \, h ) 
\end{align*}
and therefore
\begin{align}
\label{proof-psi-derivative}
 \psi^{\prime}(v) h = \mathcal{G}_{v} (h - 2 \beta \, v \, \psi(v) \, h ).
\end{align}
Note that with the H\"older inequality and the embedding of $H^1_0(\Omega)$ into $L^6(\Omega)$ we have
\begin{align*}
\| h - 2 \beta \, v \, \psi(v) \, h \| &\le \| h \| + 2 \beta \| v \|_{L^6(\Omega)} \| \psi(v) \|_{L^6(\Omega)} \| h \|_{L^6(\Omega)} \\
&\le C \left( 1+  \beta \| v \|_{H^1(\Omega)} \| \psi(v) \|_{H^1(\Omega)}  \right) \| h \|_{H^1(\Omega)} \\
&\le C \left( 1+  \beta \| v \|_{H^1(\Omega)}^2  \right) \| h \|_{H^1(\Omega)},
\end{align*}
where $C$ only depends on $\Omega$ and $d$. Hence, $(h - 2 \beta \, v \, \psi(v) \, h) \in L^2(\Omega)$ and $ \mathcal{G}_{v} (h - 2 \beta \, v \, \psi(v) \, h ) \in H^1_0(\Omega)$ is well-defined with
\begin{align*}
\| \mathcal{G}_{v} (h - 2 \beta \, v \, \psi(v) \, h )  \|_{H^1(\Omega)} \le C \| h - 2 \beta \, v \, \psi(v) \, h \| \le  C \left( 1+  \beta \| v \|_{H^1(\Omega)}^2  \right) \| h \|_{H^1(\Omega)}.
\end{align*}
Replacing $\psi(v)$ by $\mathcal{G}_{v} v$ in \eqref{proof-psi-derivative} proves \eqref{psi-derivative}.

For the Fr\'echet derivative of $\phi(v)$, we easily calculate that
\begin{align*}
\phi^{\prime}(u) 
= \frac{\psi^{\prime}(u)}{ \| \psi(u) \| } - \frac{1}{ \| \psi(u) \|^3 }
(\psi^{\prime}(u) , \psi(u))_{L^2(\Omega)}\, \psi(u).
\end{align*}
Replacing $\psi(v)=\mathcal{G}_{v} v$ and $\psi^{\prime}(v) h= \mathcal{G}_{v} (h - 2 \beta \, v \, \psi(v) \, h )$ in this formula finishes the proof.
\end{proof}
With Lemma \ref{lemma-1}, we can now calculate $\phi^{\prime}(u)$ for the ground state $u$ by exploiting that $\mathcal{G}_{u} u = \mathcal{A}_u^{-1} \mathcal{I} u = \lambda^{-1} u$ and $\| u \|=1$. We obtain the following result.

\begin{conclusion}
\label{frechet-derivatives-psi-phi-in-u}
Let $u \in H^1_0(\Omega)$ denote the ground state given by \eqref{groundstate-ev-problem} and $h\in H^1_0(\Omega)$ an arbitrary function. In the setting of Lemma \ref{lemma-1} we obtain for the Fr\'echet derivatives in $u$ that
\begin{align}
\label{psi-derivative-u}
 \psi^{\prime}(u) h = \mathcal{G}_{u} ( \hspace{1pt} (1 - 2 \tfrac{\beta}{\lambda} \, |u|^2) h \hspace{1pt} )
\end{align} 
and
\begin{align*}
 \phi^{\prime}(u) h
&= \lambda \, \mathcal{G}_{u} ( \hspace{1pt} (1 - 2 \tfrac{\beta}{\lambda} \, |u|^2) h \hspace{1pt} ) - \lambda \,
( \, \mathcal{G}_{u} ( \hspace{1pt} (1 - 2  \tfrac{\beta}{\lambda} \, |u|^2) h \hspace{1pt} ) \, , u \, )_{L^2(\Omega)}\, u.
\end{align*}
\end{conclusion}

\subsection{Proof of Theorem \ref{thm_conv_basic_inverse_iteration}}
\label{subsection-proof-main-1}
In the light of Conclusion \ref{frechet-derivatives-psi-phi-in-u}, it is apparent that we need to get control over the term $(1 - 2 \tfrac{\beta}{\lambda} \, |u|^2)$ in order to be able to estimate the spectral radius of $\phi^{\prime}(u)$. The following lemma is crucial for that and it will be later also the key to understand the effect of spectral shifts on the method.

\begin{lemma}\label{lemma-Linfty-theta-u}
Assume \ref{A1}-\ref{A3} and let $u$ denote as usual the unique positive ground state given by \eqref{groundstate-ev-problem} and let $\lambda>0$ denote the corresponding eigenvalue. Then it holds
\begin{align}
\label{Linfty-theta-u}
\| 1 -2 \tfrac{\beta}{\lambda} |u|^2 \|_{L^{\infty}(\Omega)} = 1.
\end{align}
\end{lemma}
\begin{proof}
Let us first assume that $V\in C^{\infty}(\overline{\Omega})$ with $V\ge 0$. In this case, we have by standard regularity arguments that $u \in C^2(\Omega)$ and consequently pointwise for all $x \in \Omega$
\begin{align*}
\lambda \, u(x) = - \tfrac{1}{2}\Delta u(x) + V(x) u(x) + \beta |u(x)|^2 u(x).
\end{align*}
Recalling that $u>0$ in $\Omega$ and $u=0$ on $\partial \Omega$ we know that $u$ has an interior maximum. Hence, evaluating the equation in the point $x^{\ast} \in \Omega$ where $u(x^{\ast})=\| u \|_{L^{\infty}(\Omega)}$, we have $u(x^{\ast})>0$ and $-\Delta u(x^{\ast})\ge 0$ and consequently
\begin{align*}
\lambda \, = \tfrac{1}{2}\tfrac{- \Delta u(x^{\ast})}{u(x^{\ast})} + V(x^{\ast}) + \beta |u(x^{\ast})|^2 \ge \beta \| u \|_{L^{\infty}(\Omega)}^2 
\quad
\Rightarrow \quad
0\le \tfrac{\beta}{\lambda} \| u \|_{L^{\infty}(\Omega)}^2 \le 1.
\end{align*}
Together with $u=0$ on $\partial \Omega$, we have proved equation \eqref{Linfty-theta-u} for smooth potentials. In the next step we extend the result by a density argument.\\[0.2em]
From now on, we consider a potential $V\in L^{\infty}(\Omega)$ which is non-negative almost everywhere and we let $(V_k)_{k\in \mathbb{N}} \subset C^{\infty}(\overline{\Omega})$ be an approximating sequence of non-negative potentials such that $\| V - V_k \|_{L^2(\Omega)} \rightarrow 0$ for $k\rightarrow \infty$. For $v\in H^1_0(\Omega)$, we define the corresponding energy functional by
$$
E_k(v) := \frac{1}{2} \int_{\Omega} \tfrac{1}{2}|\nabla v|^2 + V_k |v|^2 + \tfrac{\beta}{2} |v|^4 \dx
$$
and the corresponding ground state is
$$
\tilde{u}_k := \mbox{arg\hspace{1pt}min} \{ E_k(v) \, | \hspace{2pt} v\in H^1_0(\Omega), \, \| v\|_{L^2(\Omega)}=1 \}.  
$$
Since $E_k(v) \rightarrow E(v)$ for all $v \in H^1_0(\Omega)$, we can select an $L^2$-normalized sequence $(v_k)_{k \in \mathbb{N}} \subset H^1_0(\Omega)$ that converges strongly in $H^1(\Omega)$ to $u$. This implies
\begin{align}
\label{limsupEkuk}
\limsup_{k\rightarrow \infty} E_k(\tilde{u}_k) \le \lim_{k\rightarrow \infty} E_k(v_k) = E(u).
\end{align}
On the contrary, since $E$ is weakly lower semi-continuous and $\tilde{u}_k$ a bounded sequence in $H^1(\Omega)$, we also have
\begin{align}
\label{liminfEkuk}\liminf_{k\rightarrow \infty} E_k(\tilde{u}_k) \ge E(u).
\end{align}
Combining \eqref{limsupEkuk} and \eqref{liminfEkuk} we see that $\lim_{k\rightarrow \infty} E_k(\tilde{u}_k)$ exists with
\begin{align}
\label{limit-Ek-vk}\lim_{k\rightarrow \infty} E_k(\tilde{u}_k) = E(u).
\end{align}
Since $\tilde{u}_k$ is uniformly bounded in $H^1(\Omega)$, we conclude the existence of a subsequence that converges strongly in $L^2(\Omega)$ and weakly in $H^1(\Omega)$ to some limit $\tilde{u}$. However, thanks to \eqref{limit-Ek-vk} we know that the convergence in $H^1(\Omega)$ must be in fact strong and it holds $E(\tilde{u})=E(u)$. By uniqueness of a positive minimizer of $E$ (cf. Proposition \ref{proposition-minimization-problem}), we conclude that $u=\tilde{u}$ and the full sequence $\tilde{u}_k$ must converge to this limit. As a direct conclusion, we also have with the embedding $H^1(\Omega) \hookrightarrow L^{6}(\Omega)$ (for $d\le 3$) that
$$
\tilde{\lambda}_k = 2 E_k(\tilde{u}_k) + \tfrac{\beta}{2} \int_{\Omega} |\tilde{u}_k|^4 \dx \hspace{5pt} \overset{k \rightarrow \infty}{\longrightarrow} \hspace{5pt} 2 E(u) + \tfrac{\beta}{2} \int_{\Omega} |u|^4 \dx = \lambda.
$$
With this, we can use the strong convergence in the Euler-Lagrange equations
\begin{align*}
- \tfrac{1}{2}\Delta \tilde{u}_k +  V_k \tilde{u}_k  +  \beta |\tilde{u}_k|^2 \tilde{u}_k = \tilde{\lambda}_k \, \tilde{u}_k
\qquad \mbox{and}
\qquad
- \tfrac{1}{2}\Delta u +  V u  +  \beta |u|^2u  = \lambda \, u,
\end{align*}
to see that $\| \Delta(\tilde{u}_k  - u ) \|_{L^2(\Omega)}\rightarrow 0$. Since $u\in H^2(\Omega)$, we finally have with the Sobolev embedding $H^2(\Omega) \hookrightarrow L^{\infty}(\Omega)$ (for $d\le 3$) that
\begin{align*}
\| \tilde{u}_k  - u \|_{L^{\infty}(\Omega)}
\le C \| \tilde{u}_k  - u  \|_{H^2(\Omega)}
\le C \| \Delta(\tilde{u}_k  - u ) \|_{L^2(\Omega)} \rightarrow 0 \quad
\mbox{for } k\rightarrow 0.
\end{align*}
Since we already know from the first part of the proof that
\begin{align*}
\| 1 -2 \tfrac{\beta}{\tilde{\lambda}_k} |\tilde{u}_k|^2 \|_{L^{\infty}(\Omega)} = 1,
\end{align*}
we can now pass to the limit to obtain the desired result for $u$ itself. 
\end{proof}
Next, we make some initial considerations about the spectrum of $\phi^{\prime}(u)$.
\begin{lemma}
\label{lemma-2}
Assume \ref{A1}-\ref{A3} and let $u$ be the ground state. Recalling that 
$\phi^{\prime}(u): H^1_0(\Omega) \rightarrow H^1_0(\Omega)$, we consider the linear eigenvalue problem seeking $v_i \in H^1_0(\Omega)$ and $\mu_i \in \R$ such that 
\begin{align*}
\phi^{\prime}(u) v_i = \mu_i \, v_i.
\end{align*}
Then, for all eigenvalues $\mu_i \not= 0$ with eigenfunctions $v_i \in H^1_0(\Omega)$ it holds
$$
(v_i , u )_{L^2(\Omega)} = 0.
$$
\end{lemma}

\begin{remark}\label{remark-on-mu-0}
Note that all eigenvalues of $\phi^{\prime}(u)$ are indeed real. Also note that if $\mu_i=0$ is an eigenvalue of $\phi^{\prime}(u)$, then it is simple and the corresponding eigenfunction is given by $v_i=(1 - 2  \tfrac{\beta}{\lambda} \, |u|^2)^{-1}u$ (provided that this is an element of $H^1_0(\Omega)$). The latter statement will be a by-product of the proof.
\end{remark}

\begin{proof}[Proof of Lemma \ref{lemma-2}]
Let $\mu_i$ be an arbitrary eigenvalue with eigenfunction $v_i$, then $\phi^{\prime}(u) v_i =\mu_i \, v_i$ can be expressed using Conclusion \ref{frechet-derivatives-psi-phi-in-u} as
\begin{align*}
 \mathcal{G}_{u} ( \hspace{1pt} (1 - 2 \tfrac{\beta}{\lambda} \, |u|^2) v_i \hspace{1pt} ) - \,
( \, \mathcal{G}_{u} ( \hspace{1pt} (1 - 2  \tfrac{\beta}{\lambda} \, |u|^2) v_i \hspace{1pt} ) \, , u \, )_{L^2(\Omega)}\, u  = \tfrac{\mu_i}{\lambda} \, v_i.
\end{align*}
For brevity, we introduce the notation $\theta_u:=1 - 2 \tfrac{\beta}{\lambda} \, |u|^2$. Applying the bilinear form $a_u(\cdot,\cdot)$ on both sides of the equation for any $w \in H^1_0(\Omega)$ yields
\begin{align}
\label{proof-lemma-2-a}
a_u( \mathcal{G}_{u} ( \hspace{1pt} \theta_u v_i \hspace{1pt} ) , w ) - \, ( \, \mathcal{G}_{u} ( \hspace{1pt} \theta_u v_i \hspace{1pt} ) \, , u \, )_{L^2(\Omega)}\, a_u( u , w)  = \tfrac{\mu_i}{\lambda} \, a_u( v_i , w).
\end{align}
Using the definition of $\mathcal{G}_{u}$ we have
\begin{align}
\label{proof-lemma-2-b}
a_u( \mathcal{G}_{u} ( \hspace{1pt} \theta_u v_i \hspace{1pt} ) , w ) = (  \theta_u v_i \hspace{1pt}  , w )_{L^2(\Omega)}
\end{align}
and with the symmetry of $a_u(\cdot,\cdot)$ also
\begin{align}
\label{proof-lemma-2-c}
( \, \mathcal{G}_{u} ( \theta_u v_i ) \, , u \, )_{L^2(\Omega)}
= a_u( \mathcal{G}_{u}  u , \mathcal{G}_{u} ( \theta_u v_i)  )  = \lambda^{-1} (\theta_u v_i  , u )_{L^2(\Omega)}.
\end{align}
Plugging \eqref{proof-lemma-2-b} and \eqref{proof-lemma-2-c} into \eqref{proof-lemma-2-a} gives
\begin{align*}
 \tfrac{\mu_i}{\lambda} \, a_u( v_i , w) &= ( \theta_u v_i  , w )_{L^2(\Omega)} - \,
\lambda^{-1} ( \theta_u v_i , u )_{L^2(\Omega)}\, a_u( u , w) \\
&= (  \theta_u v_i , w )_{L^2(\Omega)} - \,
( \theta_u v_i , u )_{L^2(\Omega)}\, ( u , w)_{L^2(\Omega)} \\
&= ( v_i , \theta_u  [ w - ( u , w)_{L^2(\Omega)} u ] )_{L^2(\Omega)}.
\end{align*}
On the right hand side, we have the $L^2$-projection of $w \in H^1_0(\Omega)$ onto the $L^2$-orthogonal complement of $u$. Hence, selecting $w=u$, we obtain
\begin{align*}
0 = \tfrac{\mu_i}{\lambda} \, a_u( v_i , u) = \mu_i (v_i , u )_{L^2(\Omega)}. 
\end{align*}
We conclude that $ (v_i , u )_{L^2(\Omega)} =0$ whenever $\mu_i \not=0$, which proves the statement of the lemma. 

Using the previous findings, we can quickly prove the statements of Remark \ref{remark-on-mu-0}. For that, assume that $\mu_i=0$ and let $P^{\perp}(w) := w - ( u , w)_{L^2(\Omega)} u$, then we have
\begin{align*}
0 = ( v_i , \theta_u  P^{\perp}(w) )_{L^2(\Omega)} =  ( P^{\perp}(v_i \theta_u) ,  w )_{L^2(\Omega)}  
\qquad \mbox{for all } w\in H^1_0(\Omega).
\end{align*}
Consequently we obtain $P^{\perp}(v_i \theta_u) = 0$ and hence
\begin{align*}
\frac{v_i \theta_u}{\| v_i \theta_u \|}  = \frac{ (v_i \theta_u, u ) }{\| v_i \theta_u \|}  u = u.
\end{align*}
This shows that if $\mu_i=0$ is an eigenvalue (which is e.g. the case for $\beta=0$), then the only option for a corresponding eigenfunction is, up to normalization, $v_i = \theta_u^{-1}u$. This also implies that the spectrum of $\phi^{\prime}(u) $ is real, because we have
\begin{align*}
\mu_i = \frac{ a_u( \phi^{\prime}(u) v_i , v_i ) }{ a_u( v_i , v_i ) } = \frac{ ( (\lambda - 2  \beta |u|^2) v_i , v_i  )_{L^2(\Omega)} }{ a_u( v_i , v_i ) },
\end{align*} 
which is real for all eigenfunctions $v_i \in H^1_0(\Omega)$ with $(u,v_i)_{L^2(\Omega)}=0$.
\end{proof}
We are now ready to prove the first main result.
\begin{proof}[Proof of Theorem \ref{thm_conv_basic_inverse_iteration}]
We assume $\beta>0$ since  Theorem \ref{thm_conv_basic_inverse_iteration} is a classical result for $\beta=0$.
Let $V^{\perp}_u := \{ v \in H^1_0(\Omega) | \hspace{2pt} (v,u)_{L^2(\Omega)}=0\}$.
As the potential eigenvalue $\mu_i=0$ is irrelevant for the spectral radius, Lemma \ref{lemma-2} shows that it is sufficient to consider $\phi^{\prime}(u)v_i = \mu_i v_i$ on $V^{\perp}_u$. Together with Conclusion \ref{frechet-derivatives-psi-phi-in-u} this shows that
\begin{align*}
\sup_{i} \frac{|\mu_i|}{|\lambda|}
= \sup_{i} \frac{|( \mathcal{G}_{u} ( \hspace{1pt} (1 - 2 \tfrac{\beta}{\lambda} \, |u|^2) v_i \hspace{1pt} ) , v_i )_{L^2(\Omega)} |}{ \| v_i \|^2 }
= \sup_{i} \frac{ |( \hspace{1pt} (1 - 2 \tfrac{\beta}{\lambda} \, |u|^2) v_i , \mathcal{G}_{u}  v_i )_{L^2(\Omega)} |}{ \| v_i \|^2 }.
\end{align*}
Recall now that Lemma \ref{lemma-Linfty-theta-u} guarantees $\| 1 -2 \tfrac{\beta}{\lambda} |u|^2 \|_{L^{\infty}(\Omega)} = 1$ and that Proposition \ref{proposition-minimization-problem} and assumptions \ref{A2} and \ref{A4} ensure that the set of points $x \in \Omega$, where $|1 - 2 \tfrac{\beta}{\lambda} \, |u(x)|^2|=1$ is a null set in $\mathbb{R}^d$ (for $\beta>0$). In fact, $1 - 2 \tfrac{\beta}{\lambda} \, |u(x)|^2=1$ is only possible for $x \in \partial \Omega$ and $1 - 2 \tfrac{\beta}{\lambda} \, |u(x)|^2=-1$ only on sets of $\R^d$ with measure zero. 
To see the latter, assume that there exists an open ball $B_{\eps}(x_0) \subset \Omega$ such that $\beta |u(x)|^2=\lambda$ for all $x\in B_{\eps}(x_0)$. Then $u(x)= \sqrt{\tfrac{\lambda}{\beta}}>0$ is constant on $B_{\eps}(x_0)$. From the eigenvalue problem we obtain with $\Delta u(x)=0$ that $V(x)  \sqrt{\tfrac{\lambda}{\beta}} = 0$
for all $x\in B_{\eps}(x_0)$. For this to be fulfilled, we would require $V=0$ a.e. in $B_{\eps}(x_0)$, which is however a contradiction to \ref{A4}.
Considering now an eigenfunction $v_i$ such that $|\mu_i|$ becomes maximal, we conclude the existence of a constant $0<\delta=\delta(v_i,u) < 1$ with
$$
 |( \hspace{1pt} (1 - 2 \tfrac{\beta}{\lambda} \, |u|^2) v_i , \mathcal{G}_{u}  v_i )_{L^2(\Omega)} | \le (1-\delta)
 |( \hspace{1pt} |v_i| , |\mathcal{G}_{u}  v_i| )_{L^2(\Omega)}.$$
With the H\"older inequality we obtain
\begin{align*}
\sup_{i} \frac{|\mu_i|}{|\lambda|}
\le (1-\delta) \sup_{i} \frac{ \| \mathcal{G}_{u} v_i \| }{ \| v_i \| }
\le (1-\delta) \sup_{v \in V^{\perp}_u} \frac{\| \mathcal{G}_{u} v \| }{ \| v \| }.
\end{align*}
Since $\mathcal{G}_u$ is a linear, compact, self-adjoint and positive operator on $L^2(\Omega)$, the Courant--Fischer min-max principle guarantees that 
\begin{align*}
 \sup_{v \in V^{\perp}_u} \frac{\| \mathcal{G}_{u} v \| }{ \| v \| } = 
 \sup_{v \in V^{\perp}_u} \frac{( \mathcal{G}_{u} v , v) }{ \| v \|^2 }  
 = \lambda_2^{-1},
\end{align*}
where $\lambda_2^{-1}$ is the second largest eigenvalue of $ \mathcal{G}_{u}$, or respectively, $\lambda_2$ is the second smallest eigenvalue of $a_u ( u_i , v ) = \lambda_i \, ( u_i , v )_{L^2(\Omega)}$ (cf. \eqref{linearized-ev-problem}). Here we also used that $\lambda_1=\lambda$ is the smallest eigenvalue with corresponding eigenfunction $u_1=u$ (see Proposition \ref{proposition-linearized-problem}). 
We conclude that the spectral radius $\rho$ of $\phi^{\prime}(u)$ can be bounded by
\begin{align*}
\rho = \sup_{i}  |\mu_i| \le (1-\delta) \tfrac{\lambda_1}{\lambda_2} < 1.
\end{align*}
The Ostrowski theorem, i.e., Proposition \ref{prop-ostrowski}, finishes the proof if we select $\eps<\delta \tfrac{\lambda_1}{\lambda_2}$.
\end{proof}

\section{Convergence rates for GFDN iterations}
\label{section-GFDN}

A popular alternative formulation of the inverse iteration is the so-called {\it discrete normalized gradient flow} (GFDN) which was first systematically studied for the Gross-Pitaevskii equation by Bao and Du \cite{BaD04} and has received a lot of attention since then. The derivation presented in \cite{BaD04} is based on a (continuous) $L^2$-gradient flow with discrete normalization, which seeks $z(\cdot,t) \in H^1_0(\Omega)$ such that
\begin{align}
\label{CGFDN}
\partial_t z(x,t) &= - E^{\prime}(\hspace{1pt}z(x,t)\hspace{1pt}) \quad \mbox{for } x\in \Omega \hspace{4pt}\mbox{ and }\hspace{4pt} t\in(t_n,t_{n+1}),
\end{align}
where $(t_n,t_{n+1}) \subset \R_{>0}$ are given open time intervals and where $z(x,t)$ is $L^2$-normalized at all discrete times $t_n$ via
\begin{align*}
z(x,t_{n+1}) := \lim_{0<\delta \to 0 } \frac{ z(x,t_{n+1} - \delta )}{\| z(x,t_{n+1} - \delta) \|}.
\end{align*}
Hence, the \quotes{initial values} for each step \eqref{CGFDN} are given by $z(x,t_n)$. Provided that $z(x,0)\in H^1_0(\Omega)$ is chosen appropriately (e.g. as a strictly positive function), it can be expected that $z(\cdot,t) \rightarrow u$ for $t\rightarrow \infty$, where $u$ is again the positive ground state.

Choosing a uniform time step size $\tau>0$ with $t_{n+1} = t_n+ \tau$, calculating $E^{\prime}(v)$ according to \eqref{Eprime} and discretizing  \eqref{CGFDN} with a semi-explicit backward Euler method, we obtain the GFDN iterations with
$$
\tfrac{1}{\tau}(u^{n+1} - u^n , v )_{L^2(\Omega)} = - a_{u^n}( u^{n+1} , v ) \qquad \mbox{for all } v\in H^1_0(\Omega).
$$
Using the previous notation $\mathcal{A}_v w := a_v(w , \cdot )$ and $\mathcal{I} w :=(w,\cdot)_{L^2(\Omega)}$, we can present the GFDN iterations compactly in the following definition.
\begin{definition}[GFDN iteration for the GPE]
Given an initial value $u^0 \in H^1_0(\Omega)$ with $\| u^0 \|=1$, the GFDN iterations are recursively given by
\begin{align}
\label{GFDM-iterations}
u^{n+1} := \frac{( \mathcal{I} + \tau \mathcal{A}_{u^n})^{-1} \mathcal{I} u^n }{\| ( \mathcal{I} + \tau \mathcal{A}_{u^n})^{-1} \mathcal{I} u^n \|}
 \qquad \mbox{for } n\in \mathbb{N}.
\end{align}
\end{definition}
Since this is just a scaled version of the basic inverse iteration \eqref{basic-inverse-iterations} together with a positive shift, our previous considerations are applicable and lead to the following local convergence result.
\begin{theorem}[Local convergence of the GFDN iterations]
\label{thm_conv_DNGF_iteration}
Assume \ref{A1}-\ref{A4}, let $u \in H^1_0(\Omega)$ denote the ground state to the GPE \eqref{groundstate-ev-problem} and let $u^n\in H^1_0(\Omega)$ denote the GFDN iterations given by \eqref{GFDM-iterations} for an arbitrary step size $\tau>0$. Then there is a neighborhood $S \subset H^1_0(\Omega)$ of $u$ and a constant $C>0$ such that
$$
\|  u^n - u \|_{H^1(\Omega)} \le C \hspace{2pt} |\tfrac{1+ \tau \lambda_1}{1 + \tau \lambda_2} |^n \hspace{2pt} \| u^0 - u \|_{H^1(\Omega)}
$$
for all starting values $u^0 \in S$
 and $n\ge  1$. Again, $\lambda_1$ and $\lambda_2$ are the first and the second eigenvalue of problem \eqref{linearized-ev-problem}.
\end{theorem}
Before proving the theorem, let us quickly discuss the result. Obviously, the linear convergence rates degenerate to $1$ for $\tau\rightarrow 0$. Hence, selecting $\tau$ small is not a good idea. On the other hand, we have
\begin{align*}
\lim_{\tau \rightarrow \infty} \frac{1+ \tau \lambda_1}{1 + \tau \lambda_2} = \frac{\lambda_1}{\lambda_2},
\end{align*}
i.e., we recover the convergence rates for the basic inverse iterations when $\tau$ is chosen sufficiently large. In particular, the rates improve with increasing $\tau$, which is consistent with the empirical observation \cite{HeP20} that the GFDN works best when $\tau$ is selected as large as possible (but small enough to prevent an observable influence from rounding errors). {\it Asymptotically}, the GFDN cannot be expected to perform better than the basic inverse iteration. However, in preasymptotic regimes away from $u$ we cannot make any general predictions and the GFDN can perform better for suitable values of $\tau$, depending on the considered setting.

We shall now prove Theorem \ref{thm_conv_DNGF_iteration} by briefly sketching the changes compared to the previous setting.

\begin{proof}[Proof of Theorem \ref{thm_conv_DNGF_iteration}]
We proceed as in the proof of Theorem \ref{thm_conv_basic_inverse_iteration}, with the difference that we define $\mathcal{G}_{v,\tau}w :=( \mathcal{I} + \tau \mathcal{A}_{v})^{-1} \mathcal{I} w$ and $\phi : H^1_0(\Omega) \rightarrow H^1_0(\Omega)$ by
\begin{align*}
\phi(v) := \frac{\mathcal{G}_{v,\tau}v}{\| \mathcal{G}_{v,\tau}v \|}.
\end{align*}
Again, we need to estimate the spectral radius of $\phi^{\prime}(u)$. Since the ground state $u \in H^1_0(\Omega)$ of \eqref{groundstate-ev-problem} is also the ground state eigenfunction to the scaled and shifted eigenvalue problem given by
\begin{align*}
a_{u,\tau}(u , v) := \tau a_u( u , v ) + (u,v)_{L^2(\Omega)} = (1+ \tau \lambda) \, ( u , v)_{L^2(\Omega)},
\end{align*}
we can repeat the arguments for the modified bilinear form $a_{u,\tau}(\cdot , \cdot)$ and with the corresponding ground state eigenvalue $(1+ \tau \lambda)$. We find that the eigenvalues $\mu_i \in \mathbb{R}$ and eigenfunctions $v_i \in H^1_0(\Omega)$ to $\phi^{\prime}(u) v_i =\mu_i \, v_i$ can be expressed as
\begin{align*}
 \mathcal{G}_{u,\tau} ( \hspace{1pt} (1 - 2 \tfrac{\tau \beta}{1+\tau \lambda} \, |u|^2) v_i \hspace{1pt} ) - \,
( \, \mathcal{G}_{u,\tau} ( \hspace{1pt} (1 - 2  \tfrac{\tau \beta}{1+ \tau \lambda} \, |u|^2) v_i \hspace{1pt} ) \, , u \, )_{L^2(\Omega)}\, u  = \tfrac{\mu_i}{1+ \tau \lambda} \, v_i.
\end{align*}
As before, we find that $\mu_i (v_i , u )_{L^2(\Omega)} = 0 $ and consequently
\begin{align*}
 (v_i , u )_{L^2(\Omega)} = 0 \hspace{5pt} \mbox{for all } \mu_i \not=0.
\end{align*}
Recalling $V^{\perp}_u = \{ v \in H^1_0(\Omega) | \hspace{2pt} (v,u)_{L^2(\Omega)}=0\}$, we obtain
\begin{align*}
\sup_{i} \frac{|\mu_i|}{|1+\tau \lambda|}
&= \sup_{i} \frac{ |( \hspace{1pt} (1 - 2 \tfrac{\beta \tau}{1+\tau \lambda} \, |u|^2) v_i , \mathcal{G}_{u,\tau}  v_i )_{L^2(\Omega)} |}{ \| v_i \|^2 }  \\
&\le \| 1 - 2 \tfrac{\beta \tau}{1+\tau \lambda} \, |u|^2 \|_{L^{\infty}(\Omega)} \sup_{v \in V^{\perp}_u}  \frac{ (\mathcal{G}_{u,\tau} v ,  v )_{L^2(\Omega)} }{ \| v \|^2 } \\
&\le \| 1 - 2 \tfrac{\beta\tau}{1+\tau \lambda} \, |u|^2 \|_{L^{\infty}(\Omega)} |1+ \tau \lambda_2|^{-1}.
\end{align*}
It remains to show that $  \| 1 - 2 \tfrac{\beta\tau}{1+\tau \lambda} \, |u|^2 \|_{L^{\infty}(\Omega)}=1$ to apply the Ostrowski theorem. For that we note that (for fixed $x\in \overline{\Omega}$)
\begin{align*}
g(\tau) := \tfrac{ 2 \beta \tau }{1+\tau \lambda} |u(x)|^2 
\end{align*}
is a monotonically increasing function with
\begin{align*}
\lim_{\tau \rightarrow 0} g(\tau) = 0 \qquad \mbox{and} \qquad \lim_{\tau \rightarrow \infty} g(\tau) = \tfrac{2\beta}{\lambda}  |u(x)|^2 \overset{\eqref{Linfty-theta-u}}{\le} 2.
\end{align*}
Consequently $| 1 - 2 \tfrac{\beta\tau}{1+\tau \lambda} \, |u(x)|^2 | \le 1$, where equality is obtained for $x \in \partial \Omega$. We conclude that 
$|\mu_i| \le \tfrac{1+ \tau \lambda_2}{1+ \tau \lambda_1}$ and therefore Proposition \ref{prop-ostrowski} finishes the proof. The $\eps$-dependency that formally enters through the application of the Ostrowski theorem can be removed as in the proof of Theorem \ref{thm_conv_basic_inverse_iteration}.
\end{proof}

\section{Convergence rates for iterations with shift}
\label{section-inverse-iteration-with-shift}

The convergence of the inverse iterations can be significantly amplified in linear settings, by introducing a spectral shift $\sigma$ that leads to rates of the form $\tfrac{|\lambda_1 - \sigma|}{|\lambda_j -\sigma|}$ for some $j\not= 1$, i.e., the closer $\sigma$ is to $\lambda_1$, the faster the convergence of the iterations. However, it was empirically observed that the generalized inverse iteration does not react favourably to spectral shifts when applied to the Gross--Pitaevskii eigenvector problem (cf. \cite{JarKM14,AHP21NumMath}). Our analysis now reveals why this is the case. For that we start with formulating the generalized inverse iteration with shift.
\begin{definition}[Shifted inverse iteration for the GPE]
Let $\sigma \in \R \setminus \{ \lambda \} $ be a shift parameter such that $\mathcal{A}_{u} - \sigma  \mathcal{I}$ has a bounded inverse and such that all of the iterates below are well-defined. 
Given an initial value $u^0 \in H^1_0(\Omega)$ with $\| u^0 \|=1$, the inverse iterations with shift are recursively given by
\begin{align}
\label{shifted-inverse-iterations}
u^{n+1} := \frac{(\mathcal{A}_{u^n} - \sigma \,  \mathcal{I} )^{-1} \mathcal{I} u^n }{\| (\mathcal{A}_{u^n} - \sigma \,  \mathcal{I})^{-1} \mathcal{I} u^n \|}
 \qquad \mbox{for } n\in \mathbb{N}.
\end{align}
\end{definition}
Let us consider the  corresponding fixed-point function $\phi : H^1_0(\Omega) \rightarrow H^1_0(\Omega)$ given by
\begin{align*}
\phi(v):= \frac{(\mathcal{A}_{v} - \sigma \,  \mathcal{I} )^{-1} \mathcal{I} v }{\| (\mathcal{A}_{v} - \sigma \,  \mathcal{I})^{-1} \mathcal{I} v \|}.
\end{align*}
As the convergence rate is given by the spectral radius $\rho$ of $\phi^{\prime}(u)$, we investigate the eigenvalue problem seeking $\mu_i \in \R$ and $v_i \in H^1_0(\Omega)$ with
$$
\phi^{\prime}(u) v_i = \mu_i \, v_i.
$$
 Proceeding as before, we find that we can characterise the spectral radius by
 \begin{align*}
 \rho = \sup_{i} |\mu_i| = |\lambda - \sigma|  \sup_{i} \frac{ |( \hspace{1pt} (1 - 2 \tfrac{\beta}{\lambda-\sigma} \, |u|^2) v_i , (\mathcal{A}_{u} - \sigma \,  \mathcal{I} )^{-1} \mathcal{I}   v_i )_{L^2(\Omega)} |}{ \| v_i \|^2 }.
 \end{align*}
The apparent issue with this characterization is that when $\sigma$ is chosen such that $|\lambda-\sigma|$ is close to zero (as usually desired for a suitable shift and accelerated convergence), then the weighting function $(1 - 2 \tfrac{\beta}{\lambda-\sigma} \, |u|^2)$ in the eigenvalue problem is exploding. To make the issue more clear: if we would use the same arguments as for the basic inverse iteration and the GFDN, then we would end up with the estimate
 \begin{align*}
 \rho \le  \frac{|\lambda - \sigma|}{|\lambda_j - \sigma|} \,  \| 1 - 2 \tfrac{\beta}{\lambda -\sigma} \, |u|^2 \|_{L^{\infty}(\Omega)},
 \end{align*}
 where $\lambda_j \not= \lambda = \lambda_1$ is the eigenvalue of problem \eqref{linearized-ev-problem} that is closest to $\sigma$ (aside from $\lambda_1$ itself). As $\| 1 - 2 \tfrac{\beta}{\lambda-\sigma} \, |u|^2 \|_{L^{\infty}(\Omega)} \lesssim |\lambda-\sigma|^{-1}$ for $ |\lambda-\sigma|\rightarrow 0$, we see that $\rho$ is expected to behave as $|\lambda_j - \sigma|^{-1}$, which can easily become larger than $1$ and all convergence is lost. Hence, the size of the function $1 - 2 \tfrac{\beta}{\lambda-\sigma} \, |u|^2$ is an essential ingredient to ensure the convergence of the generalized inverse iterations. By shifting too strong in the direction of $\lambda$ we lose control over the weighting function and face a potential blow-up. This gives an analytical justification for the previously numerically observed phenomenon that the shifted inverse iterations do not work well (or at all) for the GPE  \cite{AHP21NumMath}. A strategy for how this can be fixed with an alternative approach (J-method) was proposed in \cite{JarKM14}.

\section{Convergence rates for iterations with damping}
\label{section-inverse-iteration-with-damping}
In this section we investigate a damped version of the inverse iteration that can be used to ensure global convergence to the ground state. An ad-hoc formulation of such a damped iteration would seek a damping parameter $\tau_n\in (0,1]$ such that the iterations
\begin{align*}
u^{n+1} := \frac{ (1-\tau_n) u^n + \tau_n \mathcal{G}_{u^n} (u^{n})}{\| (1-\tau_n) u^n + \tau_n  \mathcal{G}_{u^n} (u^{n}) \|} \end{align*}
are converging for any starting value. It can be expected that the above iterations are again locally converging to $u$ for any fixed $\tau_n=\tau \in (0,1]$ with (at least) the rate $|1-\tau|+\tau \tfrac{\lambda_1}{\lambda_2}$. However, the above ad-hoc method is unfortunately not sufficient to prove indeed {\it global} convergence and a slight modification is necessary. In this section we will present a suitable damping strategy that was first suggested in \cite{HeP20} and which can be interpreted as a gradient method in the context of Riemannian optimization where the inner product changes in each iteration. We shall recall the global convergence result proved in \cite{HeP20} and then present our final main result, which establishes explicit asymptotic rates for the convergence.

\subsection{Discrete projected Sobolev gradient flow and its convergence}
The original derivation presented in \cite{HeP20} is based on a particular Sobolev-gradient flow on the $L^2$-sphere in $H^1_0(\Omega)$, i.e., on the manifold $\mathbb{S}:=\{ v\in H^1_0(\Omega) | \hspace{3pt} \| v \|=1 \} $. Recalling that $\mathcal{G}_{v} v =\mathcal{A}_v^{-1} \mathcal{I}$, the gradient flow seeks $z \in C^1([0,\infty);\mathbb{S})$ with $z(0)=u_0 \in \mathbb{S}$ such that, for all $t>0$,
\begin{align}
\label{sobolev-flow-az}
z^{\prime}(t) = - z(t) +  \gamma(z(t)) \, 
\mathcal{G}_{z(t)} z(t), \qquad \mbox{where }
\gamma(z) :=  \frac{(z,z)_{L^2(\Omega)}}{a_{z}(\mathcal{G}_{z} z ,\mathcal{G}_{z} z)} >0.
\end{align}
Global well-posedness of $z$ was established in \cite[Theorem 3.2]{HeP20}.
To convince ourselves that this is a reasonable gradient flow, let us exemplarily sketch that the flow does not leave the sphere and that it is energy diminishing with respect to the Gross-Pitaevskii functional $E$ given by \eqref{definition-energy-functional}. For the first property, we compute
\begin{align*}
\tfrac{1}{2} \tfrac{d}{dt} \| z(t) \|^2
&=  (z^{\prime}(t), z(t) )_{L^2(\Omega)} \overset{\eqref{sobolev-flow-az}}{=} - ( z(t) , z(t) )_{L^2(\Omega)} + \gamma(z(t)) \, ( \mathcal{G}_{z(t)} z(t) , z(t) )_{L^2(\Omega)} \\
&=- ( z(t) , z(t) )_{L^2(\Omega)} + \gamma(z(t)) \hspace{2pt} a_{z(t)} ( \mathcal{G}_{z(t) } z(t) ,\mathcal{G}_{z(t)} z(t) )_{L^2(\Omega)} \overset{\eqref{sobolev-flow-az}}{=} 0.
\end{align*}
Hence, $\| z(t) \|$ is constant in time. In other words, starting from a point $u_0$ on the sphere $\mathbb{S}$, the gradient flow $z$ will never leave that sphere. This ensures that any limit point $u^{\ast}$ of $z(t)$ fulfills the normalization constraint $\| u^{\ast} \|=1$.

For verify that the energy is diminished, we apply the bilinear form $a_{z(t)}(\cdot,\cdot)$ to equation \eqref{sobolev-flow-az} and use $z^{\prime}(t) \in H^1_0(\Omega)$ as a test function. This yields
\begin{align*}
0 &\le a_{z(t)}( z^{\prime}(t) , z^{\prime}(t) ) = - a_{z(t)}( z(t) , z^{\prime}(t) )+ \gamma(z(t)) \, a_{z(t)}( \mathcal{G}_{z(t)} z(t) , z^{\prime}(t) )\\
&= - a_{z(t)}( z(t) , z^{\prime}(t) )+ \gamma(z(t)) \, ( z(t) , z^{\prime}(t) )_{L^2(\Omega)} = - \tfrac{d}{dt} E(\hspace{1pt}z(t)\hspace{1pt}).
\end{align*}
In the last step we used that $\langle E'(z) , z' \rangle = a_z(z,z')$ and $(z,z^{\prime})_{L^2(\Omega)}=0$. The calculation shows that $E(z(t))$ is monotonically decreasing with $t$ and we can expect convergence to a minimizer (or, in general, critical point) of $E$. Even more, it was shown in \cite{HeP20} that if the limit point is the ground state $u$, then the rate of convergence can be explicitly stated: for all $0<\eps \le 1$,  there exists a constant $c_{\eps}>0$ and a finite time $0<t_{\eps}<\infty$, such that for all $t\ge t_{\eps}$
\begin{align}
\label{convergence-grad-flow}
\| z(t) - u \|_{H^1(\Omega)} \le \hspace{2pt} c_{\eps} \operatorname{exp}\left(-\left(1 - \tfrac{\lambda_{1}}{\lambda_2}  - \eps \right)t\right).
\end{align}
Here, $\lambda_1$ and $\lambda_2$ are again given by \eqref{linearized-ev-problem}. Hence, the above convergence result for $z(t)$ has the same flavour as the local convergence result in Theorem \ref{thm_conv_basic_inverse_iteration} for the inverse iteration for the GPE. We note however that the proof of the rates \eqref{convergence-grad-flow} is very different to the proof of the convergence rates for the inverse iteration presented in this paper.

The damped version of the inverse iteration is now obtained by discretizing \eqref{sobolev-flow-az} with a forward Euler method, which is justified due to the bounded spectrum of the compact operator $\mathcal{G}_{v}$. 
Note that an alternative backward Euler discretization of \eqref{sobolev-flow-az} would be highly unfeasible as this would lead to nonlinear iterations with a computationally complicated structure.

With $z^{\prime}(t^n)\approx \tfrac{z^{n+1}-z^n}{\tau_n}$, we obtain the preliminary iterations
\begin{align*}
 \tfrac{z^{n+1}-z^n}{\tau_n} = - z^n +  \gamma(z^n) \, 
\mathcal{G}_{z^n} z^n, \qquad \mbox{or equivalently } \hspace{5pt} z^{n+1} = (1-\tau_n)z^n + \tau_n \gamma(z^n) \, 
\mathcal{G}_{z^n} z^n.
\end{align*}
As the $L^2$-normalization is lost in the discretization, we renormalize after each time step and define $u^{n+1}:=\tfrac{z^{n+1}}{\| z^{n+1} \|}$. Then in the next time step, $z^{n+1}$ needs to be replaced by $u^{n+1}$. In the final formulation of the method, we also use that $\gamma(v) = \frac{(v,v)_{L^2(\Omega)}}{a_{v}(\mathcal{G}_{v} v ,\mathcal{G}_{v} v)}= (\mathcal{G}_{v} v ,v)_{L^2(\Omega)}^{-1} $ for all $v\in \mathbb{S}$. This will later simplify the computation of the corresponding Fr\'echet derivative.

In summary, we obtain the following numerical method as a damped inverse iteration.
\begin{definition}[Damped inverse iteration for the GPE]
Given an initial value $u^0 \in H^1_0(\Omega)$ with $\| u^0 \|=1$ and given a sequence of damping parameters $(\tau_n)_{n\in \mathbb{N}}$ with $0<\tau_n <2$, the iterations are recursively given by
\begin{align}
\label{damped-inverse-iterations}
u^{n+1} := \frac{ (1-\tau_n) u^n + \tau_n \, \gamma(u^n) \, \mathcal{G}_{u^n} (u^{n})}{\|  (1-\tau_n) u^n + \tau_n \, \gamma(u^n) \, \mathcal{G}_{u^n} (u^{n}) \|} \qquad \mbox{for } n\in \mathbb{N}
\end{align}
where
\begin{align*}
\gamma(v) = (\mathcal{G}_{v} v ,v)_{L^2(\Omega)}^{-1} >0.
\end{align*}
Typically, the damping parameter $\tau_n\in (0,2)$ is chosen adaptively via line search such that
\begin{align}
\label{optimal-tau-n}
\tau_n = \mbox{\rm arg\hspace{2pt}min} \left\{ E\hspace{-1pt}\left(
\frac{ (1-\tau_n) u^n + \tau_n \, \gamma(u^n) \, \mathcal{G}_{u^n} (u^{n})}{\|  (1-\tau_n) u^n + \tau_n \, \gamma(u^n) \, \mathcal{G}_{u^n} (u^{n}) \|}
\right) \hspace{2pt} | \hspace{4pt} 0<\tau<2 \right\},
\end{align}
where $E$ is the usual GP energy functional given by \eqref{definition-energy-functional}. Details on how to efficiently realize \eqref{optimal-tau-n} in an implementation are given in \cite{HeP20,APS21}. 
\end{definition}
Note that for the uniform choice $\tau_n = \tau=1$, we recover from \eqref{damped-inverse-iterations} the basic inverse iteration as defined in \eqref{basic-inverse-iterations}.

The following result shows that the inverse iteration with adaptive damping converges globally to the positive ground state, if the initial value is selected non-negative. The proof is found in \cite[Theorem 5.1]{HeP20}.
\begin{proposition}[Global convergence to the ground state]
We consider the damped inverse iterations \eqref{damped-inverse-iterations} with a sequence of parameters $(\tau_n)_{n \in \mathbb{N}}$ that fulfils
\begin{align*}
0 < \tau_{\mbox{\rm\tiny min}} \le \tau_n \le \tau_{\mbox{\rm\tiny max}} < 2,
\end{align*}
for some fixed lower and upper bounds $\tau_{\mbox{\rm\tiny min}}$ and $\tau_{\mbox{\rm\tiny max}}$. Under assumptions \ref{A1}-\ref{A3}, there exists a suitable value for $\tau_{\mbox{\rm\tiny max}}$, such that for any starting value $u^0 \in H^1_0(\Omega)$ with $\| u^0 \|=1$ and $u^{0}\ge 0$ it holds
$$
\lim_{n\rightarrow \infty} \| u^n -  u\|_{H^1(\Omega)} = 0,
$$
where $u \in H^1_0(\Omega)$ is again the unique positive ground state given by \eqref{def-ground-state} and \eqref{groundstate-ev-problem}.
\end{proposition}
As global convergence is guaranteed for all parameters $\tau \le \tau_{\mbox{\rm\tiny max}}$, it is interesting to ask what is the asymptotic convergence rate in a neighborhood of the ground state and how does it depend on $\tau$. A linear convergence rate in a neighborhood of $u$ was first established by Zhang \cite{Zha21}. However, the rate obtained in \cite{Zha21} is not explicit and only proved for sufficiently small values of $\tau$. In the following theorem we specify the asymptotic convergence rate by showing that it can be bounded by $|1-\tau| + \tau \tfrac{\lambda_1}{\lambda_2}$ and that it in fact holds for all $\tau \in ( 0 , 2 \hspace{1pt}(1+ \tfrac{\lambda_1}{\lambda_2} )^{-1} )$. In particular, linear convergence is established for every $0<\tau \le 1$ (and even slightly larger values of $\tau$).
\begin{theorem}[Asymptotic convergence rate of the inverse iterations with damping]
\label{thm_conv_inverse_iteration_damping}
Assume \ref{A1}-\ref{A4}, let $u \in H^1_0(\Omega)$ denote the ground state to the GPE \eqref{groundstate-ev-problem} and let $u^n\in H^1_0(\Omega)$ denote the inverse iterations with damping given by \eqref{damped-inverse-iterations} with a fixed damping parameter $\tau_n=\tau$, which fulfills
$$
0<\tau<  \tau_{\mbox{\tiny crit}} := 2 \, (1+ \tfrac{\lambda_1}{\lambda_2} )^{-1}, \qquad \mbox{where we note that }  1< \tau_{\mbox{\tiny crit}} < 2.
$$
Recall here that $\lambda_1$ and $\lambda_2$ are the first and the second eigenvalue of problem \eqref{linearized-ev-problem}. In this setting, there is a environment $S \subset H^1_0(\Omega)$ of $u$ and a constant $C>0$ such that
$$
\|  u^n - u \|_{H^1(\Omega)} \le C \hspace{2pt} \left( |1-\tau| + \tau \tfrac{\lambda_1}{\lambda_2} \right)^n \hspace{2pt} \| u^0 - u \|_{H^1(\Omega)}
$$
for all starting values 
$u^0 \in S$
and $n\ge  1$.
\end{theorem}
The proof is postponed to Section \ref{subsection-proof-convergence-damping}.

The rate in Theorem \ref{thm_conv_inverse_iteration_damping} becomes best for $\tau=1$, i.e., when the method coincides with the basic inverse iteration. This suggests that the iterations \eqref{damped-inverse-iterations} can be split into two phases: A first phase, where $\tau_n\in (0,2)$ is computed adaptively to ensure global convergence; and a second phase, where the damping can be potentially switched off (i.e. $\tau_n=1$) as soon as the iterates are close to the ground state. In this second phase, the convergence rate approaches at least the contraction factor $\tfrac{\lambda_1}{\lambda_2}$, known from the basic inverse iteration. However, a sharper bound for the convergence rate in the asymptotic phase (for fixed $\tau$) is given by $|1-\tau + \tau \, \mu_j|$, where $\mu_j$ is the eigenvalue of the $\tau$-independent eigenvalue problem \eqref{weighted-evp} (previously presented in Remark \ref{remark-sharper-rates}) such that the expression becomes maximal. From this more accurate formula we see that there might be space for improvements and the best {\it asymptotic rate} might now always be attained for  $\tau=1$.

\subsection{Proof of Theorem \ref{thm_conv_inverse_iteration_damping}}
\label{subsection-proof-convergence-damping}

To keep the presentation short, we define 
\begin{align*}
\phi_{\tau} ( v ):= \frac{\psi_{\tau}(v)}{\| \psi_{\tau}(v) \|}, \qquad \mbox{where } \qquad
\psi_{\tau}(v):= (1-\tau) v + \tau \gamma(v) \mathcal{G}_v v.
\end{align*}
With this, the iteration \eqref{damped-inverse-iterations} is compactly written as $u^{n+1}=\phi_{\tau}(u^n)$.
Note that for $\tau=1$, we obtain $\phi_1(v)=\phi(v)$ and $\psi_1(v) = \gamma(v) \psi(v)$, where $\phi$ and $\psi$ are defined as in Lemma \ref{lemma-1}.

To compute the Fr\'echet derivative of $\phi_{\tau}$ in $u$, we start with $\gamma$ and $\psi_{\tau}$ in the following lemma.

\begin{lemma}
The mapping $\gamma : H^1_0(\Omega) \rightarrow \mathbb{R}$ is Fr\'echet-differentiable for all $v\in H^1_0(\Omega)$. For the ground state $u\in H^1_0(\Omega)$ and a direction $h \in H^1_0(\Omega)$ we have
\begin{align*}
 \gamma^{\prime}(u) h   =   - 2 \, \lambda \, ( u - \tfrac{\beta}{\lambda} \, |u|^2 u \hspace{1pt} , h \hspace{1pt})_{L^2(\Omega)}
\end{align*}
and
\begin{align*}
 \psi_{\tau}^{\prime}(u) h  &= (1-\tau)h - 2\, \tau \, ( \hspace{1pt} (1 - \tfrac{\beta}{\lambda} \, |u|^2) u \hspace{1pt} , h )_{L^2(\Omega)} \, u +  \tau \, \lambda \,  \mathcal{G}_{u} ( \hspace{1pt} (1 - 2 \tfrac{\beta}{\lambda} \, |u|^2) h \hspace{1pt} ).
\end{align*} 
\end{lemma}

\begin{proof}
For $v\in H^1_0(\Omega)$, we compute the Fr\'echet derivative of $\gamma(v)$ in direction $h\in H^1_0(\Omega)$ as
$$
 \gamma^{\prime}(v) h  =  - \gamma(v)^2 \left( ( \, \psi'(v) h \, , v )_{L^2(\Omega)} + ( \psi(v), h )_{L^2(\Omega)} \right), 
$$
where $\psi(v)= \mathcal{G}_v v$. Evaluating this expression for the ground state $u$, we have $\psi(u) = \lambda^{-1} u$ and $\gamma(u) = \lambda$. Together with the formula $\psi^{\prime}(u) h = \mathcal{G}_{u} ( \hspace{1pt} (1 - 2 \tfrac{\beta}{\lambda} \, |u|^2) h \hspace{1pt} )$ 
from Conclusion \ref{frechet-derivatives-psi-phi-in-u}, we obtain
\begin{align*}
 \gamma^{\prime}(u) h  &=  - \lambda^2 \left( (\mathcal{G}_{u} ( \hspace{1pt} (1 - 2 \tfrac{\beta}{\lambda} \, |u|^2) h \hspace{1pt} ) , u )_{L^2(\Omega)} +  \lambda^{-1}  ( u, h )_{L^2(\Omega)} \right) \\
 &=  - \lambda \left( ( \hspace{1pt} (1 - 2 \tfrac{\beta}{\lambda} \, |u|^2) h \hspace{1pt} , u )_{L^2(\Omega)} + ( u, h )_{L^2(\Omega)} \right) \\
 &=   - 2 \, \lambda \, ( \hspace{1pt} (1 - \tfrac{\beta}{\lambda} \, |u|^2) u \hspace{1pt} , h )_{L^2(\Omega)}.
\end{align*}
Recalling $\psi_{\tau}(v)= (1-\tau) v + \tau \gamma(v) \psi(v)$ and $\psi(v)=\mathcal{G}_v v$, we further obtain
\begin{align*}
 \psi_{\tau}^{\prime}(u) h &= (1-\tau)h + \tau \, (\gamma^{\prime}(u) h) \, \mathcal{G}_{u} u +  \tau \, \gamma(u) \,  \psi^{\prime}(u) h  \\
&= (1-\tau)h - 2\, \tau \, ( \hspace{1pt} (1 - \tfrac{\beta}{\lambda} \, |u|^2) u \hspace{1pt} , h )_{L^2(\Omega)} \, u +  \tau \, \lambda \, \psi^{\prime}(u) h \\
&= (1-\tau)h - 2\, \tau \, ( \hspace{1pt} (1 - \tfrac{\beta}{\lambda} \, |u|^2) u \hspace{1pt} , h )_{L^2(\Omega)} \, u +  \tau \, \lambda \,  \mathcal{G}_{u} ( \hspace{1pt} (1 - 2 \tfrac{\beta}{\lambda} \, |u|^2) h \hspace{1pt} ).
\end{align*} 
\end{proof}

With this lemma and our previous results from Section \ref{section-basic-inverse-iteration} we are now ready to prove the final main result.

\begin{proof}[Proof of Theorem \ref{thm_conv_inverse_iteration_damping}]
As before, we need to estimate the largest eigenvalue in magnitude of $\psi_{\tau}^{\prime}(u)$ to apply the Ostrowski Theorem. Since $\psi_{\tau}(u)=(1-\tau) u + \tau \lambda \, \mathcal{G}_uu = u$, we have
\begin{align*}
\phi^{\prime}_{\tau}(u) 
&= \frac{\psi^{\prime}_{\tau}(u)}{ \| \psi_{\tau}(u) \| } - \frac{1}{ \| \psi_{\tau}(u) \|^3 }
(\psi^{\prime}_{\tau}(u) , \psi_{\tau}(u))_{L^2(\Omega)}\, \psi_{\tau}(u) \\
&= \psi^{\prime}_{\tau}(u)   -  (\psi^{\prime}_{\tau}(u) , u )_{L^2(\Omega)}\, u.
\end{align*}
This implies for any $v \in H^1_0(\Omega)$ that
\begin{align*}
( \phi^{\prime}_{\tau}(u)v , u )_{L^2(\Omega)} = ( \psi^{\prime}_{\tau}(u) v, u  )_{L^2(\Omega)}  -  (\psi^{\prime}_{\tau}(u) v , u )_{L^2(\Omega)}\, \| u \|^2 = 0.
\end{align*}
Consequently, if $v_i$ is an eigenfunction with eigenvalue $\mu_i$ we have
\begin{align*}
0 = ( \phi^{\prime}_{\tau}(u)v_i , u )_{L^2(\Omega)} = \mu_i ( v_i , u )_{L^2(\Omega)}.
\end{align*} 
As before, we conclude that $v_i$ is $L^2$-orthogonal to $u$ for all eigenvalues $\mu_i\not=0$. We can therefore restrict ourselves to seeking $\mu_i \in \mathbb{R} \setminus \{0\}$ and $v_i \in H^1_0(\Omega)$ with $(v_i,u)_{L^2(\Omega)}=0$ such that
\begin{align}
\label{ev-problem-mu-tau}
 ( \phi^{\prime}_{\tau}(u)v_i , w )_{L^2(\Omega)} =  \mu_i ( v_i , w )_{L^2(\Omega)} \qquad \mbox{for all } w\in H^1_0(\Omega) \mbox{ with } (w,u)_{L^2(\Omega)}=0.
\end{align}
On the orthogonal complement of $u$, $( \phi^{\prime}_{\tau}(u)v_i , w )_{L^2(\Omega)}$ simplifies to
\begin{align}
\label{equation-phi-prime-tau}
 ( \phi^{\prime}_{\tau}(u)v_i , w )_{L^2(\Omega)} = (1-\tau) ( v_i , w )_{L^2(\Omega)} + \tau \lambda ( \mathcal{G}_{u} ( \hspace{1pt} (1 - 2 \tfrac{\beta}{\lambda} \, |u|^2) v_i \hspace{1pt} )  , w )_{L^2(\Omega)}.
\end{align}
We already know from the proof of Theorem \ref{thm_conv_basic_inverse_iteration}, that the largest eigenvalue in magnitude of $\mathcal{G}_{u} ( \hspace{1pt} (1 - 2 \tfrac{\beta}{\lambda} \, |u|^2) (\hspace{1pt}\cdot\hspace{1pt}) \hspace{1pt} )$ on the $L^2$-orthogonal complement of $u$ can be bounded by $\lambda_2^{-1}$. The first part in \eqref{equation-phi-prime-tau} is just a spectral shift. For all $\mu_i$ we obtain therefore
\begin{align*}
|\mu_i| \le |1-\tau| + \tau |\tfrac{\lambda}{\lambda_2}| - \eps.
\end{align*}
for some $\eps>0$ if $\beta>0$. As $\lambda=\lambda_1$, Proposition \ref{prop-ostrowski} finishes the proof.
\end{proof}
In the last step, we crudely estimated $|\mu_i|$ by the sum of the absolute values of the spectral shift $1-\tau$ and the maximum eigenvalue $ \tau \tfrac{\lambda}{\lambda_2}$. With this, one might wonder if it is possible to remove the absolute values with a more careful argument to get the improved rate $|1+ \tau (\tfrac{\lambda_1}{\lambda_2}-1)|$ for $\tau\ge 1$. However, this particular rate is unfortunately impossible as seen by contradiction. If that rate would be achievable, then the choice $\tau = \frac{\lambda_2}{\lambda_2 - \lambda_1} > 1$ would allow for arbitrary fast convergence. However, numerical line search experiments for finding optimal values for $\tau$ cannot confirm this (cf. \cite{HeP20}). Furthermore, it was also proved in \cite{HeP20} that if $\tau \ge 2$, then the damped inverse iterations \eqref{damped-inverse-iterations} must necessarily diverge. Since $\frac{\lambda_2}{\lambda_2 - \lambda_1}$ can easily become larger than $2$ for problems with suitable spectral gaps, we would have another contradiction. This shows that, in general, the rate $|1+ \tau (\tfrac{\lambda_1}{\lambda_2}-1)|$ is not possible for $\tau > 1$. However, note that this does not necessarily mean that $|\mu_i|$ becomes optimal for $\tau=1$ and improvements are potentially possible.

\section{Numerical experiments}
\label{section-numerical-experiments}

In the following numerical experiments, we consider the Gross--Pitaevskii eigenvector problem in $1d$ as this simplifies the numerical study of convergence rates and since our results do not depend on the space dimension. We also restrict the numerical investigations to the basic inverse iterations as formulated in Definition \ref{definition-basic-inverse-iteration} in order to focus on the influence of spectral gaps on the asymptotic rates.

In our experiments the space discretization is based on $P1$ finite elements with $10^3$ degrees of freedom. The ground state is always computed with an accurate reference computation (in the same finite element space) using the stopping criterion that two successive iterations need to produce approximations of the eigenvalue $\lambda^{(n)}:= a_{u^{n}}(u^n,u^n)$ with $|\lambda^{(n+1)} - \lambda^{(n)}| \le 10^{-13}$. For an iteration $n$, we define the numerical contraction rate for the $H^1$-error by
\begin{align}
\label{contraction-rate-H1}
r(n) := \frac{\| u - u^{n+1}\|_{H^1(\Omega)} }{ \| u - u^n \|_{H^1(\Omega)} },
\end{align}
where $u$ is the unique positive ground state and $u^n$ is the result from the $n$'th iteration of the basic inverse iteration \eqref{basic-inverse-iterations}. All iterations are initialized with a random initial value that is normalized in $L^2(\Omega)$. The random initial value is selected to increase the required number of iterations and to get a better picture on the contraction rates. A more reasonable starting value from a practical perspective would be the ground state of the linear equation with $\beta=0$ or a Thomas-Fermi approximation.

In the numerical experiments we compare the observed contraction rate $r(n)$ with the predicted upper bound $\tfrac{\lambda}{\lambda_2}$ according to Theorem \ref{thm_conv_basic_inverse_iteration} and with the improved upper bound $|\mu_1|$ from Remark \ref{remark-sharper-rates}.

\subsection{Inverse iteration applied to GPE with moderate spectral gap}

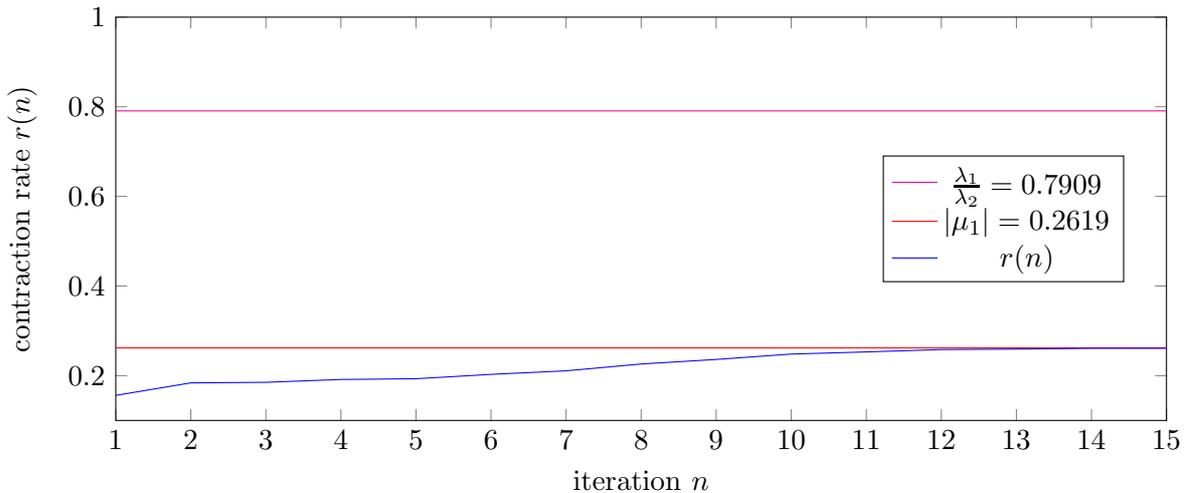
\begin{figure}
\begin{tikzpicture}
    \begin{axis}[
       legend style={at={(0.73,0.5)},anchor=west},
        height=0.45\textwidth,
        width=1.00\textwidth,
        xmax   = 15,  
        xmin   = 1, %
        ymax = 1.0, 
        ymin = 0.1,
        xlabel=iteration $n$,
        ylabel=contraction rate $r(n)$]
          \addplot[magenta] coordinates {
        ( 1 , 0.7909 ) 
        ( 2 , 0.7909 ) 
        ( 3 , 0.7909 ) 
        ( 4 , 0.7909 ) 
        ( 5 , 0.7909 ) 
        ( 6 , 0.7909 ) 
        ( 7 , 0.7909 ) 
        ( 8 , 0.7909 ) 
        ( 9 , 0.7909 ) 
      ( 10 , 0.7909 ) 
      ( 11 , 0.7909 ) 
      ( 12 , 0.7909 ) 
      ( 13 , 0.7909 ) 
      ( 14 , 0.7909 ) 
      ( 15 , 0.7909 ) };
          \addplot[red] coordinates {
        ( 1 , 0.2619 ) 
        ( 2 , 0.2619 ) 
        ( 3 , 0.2619 ) 
        ( 4 , 0.2619 ) 
        ( 5 , 0.2619 ) 
        ( 6 , 0.2619 ) 
        ( 7 , 0.2619 ) 
        ( 8 , 0.2619 ) 
        ( 9 , 0.2619 ) 
      ( 10 , 0.2619 ) 
      ( 11 , 0.2619 ) 
      ( 12 , 0.2619 ) 
      ( 13 , 0.2619 ) 
      ( 14 , 0.2619 ) 
      ( 15 , 0.2619 ) };
         \addplot[blue] coordinates {
( 1 , 0.15594 ) 
( 2 , 0.18386 ) 
( 3 , 0.18514 ) 
( 4 , 0.19146 ) 
( 5 , 0.19311 ) 
( 6 , 0.20284 ) 
( 7 , 0.21076 ) 
( 8 , 0.22600 ) 
( 9 , 0.23614 ) 
( 10 , 0.24815 ) 
( 11 , 0.25294 ) 
( 12 , 0.25820 ) 
( 13 , 0.25924 ) 
( 14 , 0.26107 ) 
( 15 , 0.26109 )};
        \legend{%
$\tfrac{\lambda_1}{\lambda_2} = 0.7909$,
$|\mu_1| = 0.2619$,
$r(n)$,
}
    \end{axis}
\end{tikzpicture}
\caption{{\it Contraction rates for model problem 1. We compare the actual contraction rate $r(n)$ at iteration $n$ with the upper bound $\tfrac{\lambda_1}{\lambda_2}$ that depends on the first spectral gap and the sharper bound $|\mu_1|$ that is obtained through the weighted eigenvalue problem  \eqref{weighted-evp}.}}
\label{rates_modelproblem1}
\end{figure}

In the first example, we consider the Gross--Pitaevskii equation on the interval $\Omega=(-2,2)$ with a potential of the form 
$$
V(x) = \tfrac{1}{4}x^2 + \sin(2 \pi x)^2 
$$
and the constant $\beta=5$ for the nonlinear term. We seek the corresponding (positive) ground state $u\in H^1_0(\Omega)$ with $\| u\|=1$ as given by \eqref{def-ground-state} and we seek the corresponding ground state eigenvalue $\lambda$. Solving for $\lambda$ with the reference computation sketched at the beginning of this section, we obtain $\lambda=\lambda_1=2.65187$. Computing the second eigenvalue of the linearized problem  \eqref{linearized-ev-problem} around the reference ground state $u$, we obtain $\lambda_2= 3.35315$. 
This leads us to the upper bound for the asymptotic contraction rate with
$$
\frac{\lambda_1}{\lambda_2} = 0.79086. 
$$
After that, we solved the weighted eigenvalue problem \eqref{weighted-evp} and obtained the largest eigenvalue in magnitude with
$$
\mu_1 = 0.26197. 
$$
We compare these rates with the results from an inverse iteration \eqref{basic-inverse-iterations} that was initialized with a random starting value. The stopping criterion was set to $|\lambda^{(n+1)} - \lambda^{(n)}| \le 10^{-11}$ to not come too close to the reference solution which would pollute the numerical contraction rates. With this stopping criterion, $n=16$ iterations were made.
The results are depicted in Figure \ref{rates_modelproblem1}. We observe that the rate $\tfrac{\lambda_1}{\lambda_2}=0.7909$ yields quite an overshoot for the observed contraction $r(n)$, which we computed according to \eqref{contraction-rate-H1}. In the final iteration, $r(n)$ approached the value $0.26109$, which is remarkably close to our guaranteed upper bound given by $|\mu_1| = 0.26197$. The considerable difference between  $\tfrac{\lambda_1}{\lambda_2}$ and  $|\mu_1|$ is also quite surprising when recalling that $\tfrac{\lambda_1}{\lambda_2}$ is the sharp rate that we would obtain when solving the linearized Gross--Pitaevskii equation (i.e. linearized around the exact ground state $u$) with a conventional (linear) inverse iteration, whereas the much faster rate $|\mu_1|$ can only be obtained from the generalized inverse iteration. In other words, a linear solver for a linear problem performs worse than a linearized solver for a nonlinear problem, both approximating the same ground state $u$.

\subsection{Inverse iteration applied to GPE with small spectral gap}

In the second test case, we regard a more tough setup with a small spectral gap after the first eigenvalue. The following problem is taken from \cite{BaC13b}.

For $\Omega=(-16,16)$, we consider the Gross--Pitaevskii equation with the harmonic oscillator potential $V(x)=\tfrac{1}{2}x^2$ and the interaction constant $\beta=400$, i.e., we seek the ground state $u\in H^1_0(\Omega)$ with $\| u\|=1$ and the smallest eigenvalue $\lambda>0$ such that
\begin{align}
\label{numexample2}
-\tfrac{1}{2} u^{\prime\prime}(x) + \tfrac{1}{2}x^2 \, u(x) + 400 \, |u(x)|^2 u(x) = \lambda\, u(x).
\end{align}
We obtain the ground state eigenvalue with $\lambda= 35.57746$, which is consistent with the findings of \cite{BaC13b}. For the second eigenvalue of the linearized problem \eqref{linearized-ev-problem} we calculated $\lambda_2=35.60994$. This yields the very poor rate
$$
\frac{\lambda_1}{\lambda_2} = 0.99909 
$$
as an upper bound for the asymptotic contraction factor according to Theorem \ref{thm_conv_basic_inverse_iteration}. To get an improved prediction for the rate in \eqref{improved-rate}, we also solve the weighted eigenvalue problem \eqref{weighted-evp} in the $L^2$-orthogonal complement of $u$. The largest eigenvalue in magnitude is found to be $\mu_1= -0.94192$.

To solve the GPE \eqref{numexample2}, we apply again the generalized inverse iteration \eqref{basic-inverse-iterations} and use the condition  $|\lambda^{(n+1)} - \lambda^{(n)}| \le 10^{-10}$ as a stopping criterion.
Initializing the process again with a random initial value, the stopping criterion was fulfilled after $333$ iterations. The results are depicted in Figure \ref{rates_modelproblem2}.

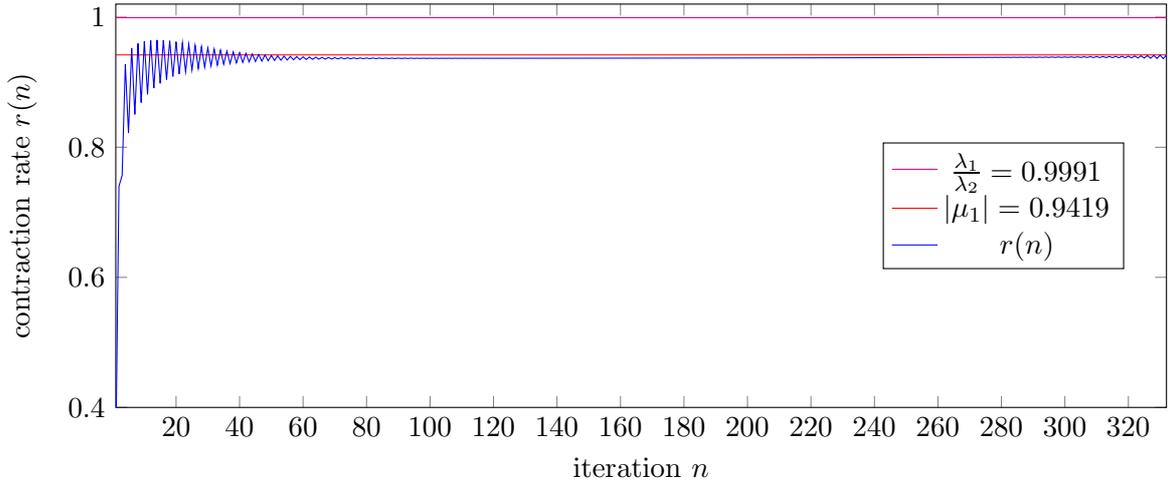
\begin{figure}
\begin{tikzpicture}
    \begin{axis}[
       legend style={at={(0.73,0.5)},anchor=west},
        height=0.45\textwidth,
        width=1.00\textwidth,
        xmax   = 332,  
        xmin   = 1, %
        ymax = 1.02, 
        ymin = 0.4,
        xlabel=iteration $n$,
        ylabel=contraction rate $r(n)$]
          \addplot[magenta] coordinates {
        ( 1 , 0.99909 ) 
( 2 , 0.99909 ) 
( 3 , 0.99909 ) 
( 4 , 0.99909 ) 
( 5 , 0.99909 ) 
( 6 , 0.99909 ) 
( 7 , 0.99909 ) 
( 8 , 0.99909 ) 
( 9 , 0.99909 ) 
( 10 , 0.99909 ) 
( 11 , 0.99909 ) 
( 12 , 0.99909 ) 
( 13 , 0.99909 ) 
( 14 , 0.99909 ) 
( 15 , 0.99909 ) 
( 16 , 0.99909 ) 
( 17 , 0.99909 ) 
( 18 , 0.99909 ) 
( 19 , 0.99909 ) 
( 20 , 0.99909 ) 
( 21 , 0.99909 ) 
( 22 , 0.99909 ) 
( 23 , 0.99909 ) 
( 24 , 0.99909 ) 
( 25 , 0.99909 ) 
( 26 , 0.99909 ) 
( 27 , 0.99909 ) 
( 28 , 0.99909 ) 
( 29 , 0.99909 ) 
( 30 , 0.99909 ) 
( 31 , 0.99909 ) 
( 32 , 0.99909 ) 
( 33 , 0.99909 ) 
( 34 , 0.99909 ) 
( 35 , 0.99909 ) 
( 36 , 0.99909 ) 
( 37 , 0.99909 ) 
( 38 , 0.99909 ) 
( 39 , 0.99909 ) 
( 40 , 0.99909 ) 
( 41 , 0.99909 ) 
( 42 , 0.99909 ) 
( 43 , 0.99909 ) 
( 44 , 0.99909 ) 
( 45 , 0.99909 ) 
( 46 , 0.99909 ) 
( 47 , 0.99909 ) 
( 48 , 0.99909 ) 
( 49 , 0.99909 ) 
( 50 , 0.99909 ) 
( 51 , 0.99909 ) 
( 52 , 0.99909 ) 
( 53 , 0.99909 ) 
( 54 , 0.99909 ) 
( 55 , 0.99909 ) 
( 56 , 0.99909 ) 
( 57 , 0.99909 ) 
( 58 , 0.99909 ) 
( 59 , 0.99909 ) 
( 60 , 0.99909 ) 
( 61 , 0.99909 ) 
( 62 , 0.99909 ) 
( 63 , 0.99909 ) 
( 64 , 0.99909 ) 
( 65 , 0.99909 ) 
( 66 , 0.99909 ) 
( 67 , 0.99909 ) 
( 68 , 0.99909 ) 
( 69 , 0.99909 ) 
( 70 , 0.99909 ) 
( 71 , 0.99909 ) 
( 72 , 0.99909 ) 
( 73 , 0.99909 ) 
( 74 , 0.99909 ) 
( 75 , 0.99909 ) 
( 76 , 0.99909 ) 
( 77 , 0.99909 ) 
( 78 , 0.99909 ) 
( 79 , 0.99909 ) 
( 80 , 0.99909 ) 
( 81 , 0.99909 ) 
( 82 , 0.99909 ) 
( 83 , 0.99909 ) 
( 84 , 0.99909 ) 
( 85 , 0.99909 ) 
( 86 , 0.99909 ) 
( 87 , 0.99909 ) 
( 88 , 0.99909 ) 
( 89 , 0.99909 ) 
( 90 , 0.99909 ) 
( 91 , 0.99909 ) 
( 92 , 0.99909 ) 
( 93 , 0.99909 ) 
( 94 , 0.99909 ) 
( 95 , 0.99909 ) 
( 96 , 0.99909 ) 
( 97 , 0.99909 ) 
( 98 , 0.99909 ) 
( 99 , 0.99909 ) 
( 100 , 0.99909 ) 
( 101 , 0.99909 ) 
( 102 , 0.99909 ) 
( 103 , 0.99909 ) 
( 104 , 0.99909 ) 
( 105 , 0.99909 ) 
( 106 , 0.99909 ) 
( 107 , 0.99909 ) 
( 108 , 0.99909 ) 
( 109 , 0.99909 ) 
( 110 , 0.99909 ) 
( 111 , 0.99909 ) 
( 112 , 0.99909 ) 
( 113 , 0.99909 ) 
( 114 , 0.99909 ) 
( 115 , 0.99909 ) 
( 116 , 0.99909 ) 
( 117 , 0.99909 ) 
( 118 , 0.99909 ) 
( 119 , 0.99909 ) 
( 120 , 0.99909 ) 
( 121 , 0.99909 ) 
( 122 , 0.99909 ) 
( 123 , 0.99909 ) 
( 124 , 0.99909 ) 
( 125 , 0.99909 ) 
( 126 , 0.99909 ) 
( 127 , 0.99909 ) 
( 128 , 0.99909 ) 
( 129 , 0.99909 ) 
( 130 , 0.99909 ) 
( 131 , 0.99909 ) 
( 132 , 0.99909 ) 
( 133 , 0.99909 ) 
( 134 , 0.99909 ) 
( 135 , 0.99909 ) 
( 136 , 0.99909 ) 
( 137 , 0.99909 ) 
( 138 , 0.99909 ) 
( 139 , 0.99909 ) 
( 140 , 0.99909 ) 
( 141 , 0.99909 ) 
( 142 , 0.99909 ) 
( 143 , 0.99909 ) 
( 144 , 0.99909 ) 
( 145 , 0.99909 ) 
( 146 , 0.99909 ) 
( 147 , 0.99909 ) 
( 148 , 0.99909 ) 
( 149 , 0.99909 ) 
( 150 , 0.99909 ) 
( 151 , 0.99909 ) 
( 152 , 0.99909 ) 
( 153 , 0.99909 ) 
( 154 , 0.99909 ) 
( 155 , 0.99909 ) 
( 156 , 0.99909 ) 
( 157 , 0.99909 ) 
( 158 , 0.99909 ) 
( 159 , 0.99909 ) 
( 160 , 0.99909 ) 
( 161 , 0.99909 ) 
( 162 , 0.99909 ) 
( 163 , 0.99909 ) 
( 164 , 0.99909 ) 
( 165 , 0.99909 ) 
( 166 , 0.99909 ) 
( 167 , 0.99909 ) 
( 168 , 0.99909 ) 
( 169 , 0.99909 ) 
( 170 , 0.99909 ) 
( 171 , 0.99909 ) 
( 172 , 0.99909 ) 
( 173 , 0.99909 ) 
( 174 , 0.99909 ) 
( 175 , 0.99909 ) 
( 176 , 0.99909 ) 
( 177 , 0.99909 ) 
( 178 , 0.99909 ) 
( 179 , 0.99909 ) 
( 180 , 0.99909 ) 
( 181 , 0.99909 ) 
( 182 , 0.99909 ) 
( 183 , 0.99909 ) 
( 184 , 0.99909 ) 
( 185 , 0.99909 ) 
( 186 , 0.99909 ) 
( 187 , 0.99909 ) 
( 188 , 0.99909 ) 
( 189 , 0.99909 ) 
( 190 , 0.99909 ) 
( 191 , 0.99909 ) 
( 192 , 0.99909 ) 
( 193 , 0.99909 ) 
( 194 , 0.99909 ) 
( 195 , 0.99909 ) 
( 196 , 0.99909 ) 
( 197 , 0.99909 ) 
( 198 , 0.99909 ) 
( 199 , 0.99909 ) 
( 200 , 0.99909 ) 
( 201 , 0.99909 ) 
( 202 , 0.99909 ) 
( 203 , 0.99909 ) 
( 204 , 0.99909 ) 
( 205 , 0.99909 ) 
( 206 , 0.99909 ) 
( 207 , 0.99909 ) 
( 208 , 0.99909 ) 
( 209 , 0.99909 ) 
( 210 , 0.99909 ) 
( 211 , 0.99909 ) 
( 212 , 0.99909 ) 
( 213 , 0.99909 ) 
( 214 , 0.99909 ) 
( 215 , 0.99909 ) 
( 216 , 0.99909 ) 
( 217 , 0.99909 ) 
( 218 , 0.99909 ) 
( 219 , 0.99909 ) 
( 220 , 0.99909 ) 
( 221 , 0.99909 ) 
( 222 , 0.99909 ) 
( 223 , 0.99909 ) 
( 224 , 0.99909 ) 
( 225 , 0.99909 ) 
( 226 , 0.99909 ) 
( 227 , 0.99909 ) 
( 228 , 0.99909 ) 
( 229 , 0.99909 ) 
( 230 , 0.99909 ) 
( 231 , 0.99909 ) 
( 232 , 0.99909 ) 
( 233 , 0.99909 ) 
( 234 , 0.99909 ) 
( 235 , 0.99909 ) 
( 236 , 0.99909 ) 
( 237 , 0.99909 ) 
( 238 , 0.99909 ) 
( 239 , 0.99909 ) 
( 240 , 0.99909 ) 
( 241 , 0.99909 ) 
( 242 , 0.99909 ) 
( 243 , 0.99909 ) 
( 244 , 0.99909 ) 
( 245 , 0.99909 ) 
( 246 , 0.99909 ) 
( 247 , 0.99909 ) 
( 248 , 0.99909 ) 
( 249 , 0.99909 ) 
( 250 , 0.99909 ) 
( 251 , 0.99909 ) 
( 252 , 0.99909 ) 
( 253 , 0.99909 ) 
( 254 , 0.99909 ) 
( 255 , 0.99909 ) 
( 256 , 0.99909 ) 
( 257 , 0.99909 ) 
( 258 , 0.99909 ) 
( 259 , 0.99909 ) 
( 260 , 0.99909 ) 
( 261 , 0.99909 ) 
( 262 , 0.99909 ) 
( 263 , 0.99909 ) 
( 264 , 0.99909 ) 
( 265 , 0.99909 ) 
( 266 , 0.99909 ) 
( 267 , 0.99909 ) 
( 268 , 0.99909 ) 
( 269 , 0.99909 ) 
( 270 , 0.99909 ) 
( 271 , 0.99909 ) 
( 272 , 0.99909 ) 
( 273 , 0.99909 ) 
( 274 , 0.99909 ) 
( 275 , 0.99909 ) 
( 276 , 0.99909 ) 
( 277 , 0.99909 ) 
( 278 , 0.99909 ) 
( 279 , 0.99909 ) 
( 280 , 0.99909 ) 
( 281 , 0.99909 ) 
( 282 , 0.99909 ) 
( 283 , 0.99909 ) 
( 284 , 0.99909 ) 
( 285 , 0.99909 ) 
( 286 , 0.99909 ) 
( 287 , 0.99909 ) 
( 288 , 0.99909 ) 
( 289 , 0.99909 ) 
( 290 , 0.99909 ) 
( 291 , 0.99909 ) 
( 292 , 0.99909 ) 
( 293 , 0.99909 ) 
( 294 , 0.99909 ) 
( 295 , 0.99909 ) 
( 296 , 0.99909 ) 
( 297 , 0.99909 ) 
( 298 , 0.99909 ) 
( 299 , 0.99909 ) 
( 300 , 0.99909 ) 
( 301 , 0.99909 ) 
( 302 , 0.99909 ) 
( 303 , 0.99909 ) 
( 304 , 0.99909 ) 
( 305 , 0.99909 ) 
( 306 , 0.99909 ) 
( 307 , 0.99909 ) 
( 308 , 0.99909 ) 
( 309 , 0.99909 ) 
( 310 , 0.99909 ) 
( 311 , 0.99909 ) 
( 312 , 0.99909 ) 
( 313 , 0.99909 ) 
( 314 , 0.99909 ) 
( 315 , 0.99909 ) 
( 316 , 0.99909 ) 
( 317 , 0.99909 ) 
( 318 , 0.99909 ) 
( 319 , 0.99909 ) 
( 320 , 0.99909 ) 
( 321 , 0.99909 ) 
( 322 , 0.99909 ) 
( 323 , 0.99909 ) 
( 324 , 0.99909 ) 
( 325 , 0.99909 ) 
( 326 , 0.99909 ) 
( 327 , 0.99909 ) 
( 328 , 0.99909 ) 
( 329 , 0.99909 ) 
( 330 , 0.99909 ) 
( 331 , 0.99909 ) 
( 332 , 0.99909)};
          \addplot[red] coordinates {
        ( 1 , 0.94193 ) 
( 2 , 0.94193 ) 
( 3 , 0.94193 ) 
( 4 , 0.94193 ) 
( 5 , 0.94193 ) 
( 6 , 0.94193 ) 
( 7 , 0.94193 ) 
( 8 , 0.94193 ) 
( 9 , 0.94193 ) 
( 10 , 0.94193 ) 
( 11 , 0.94193 ) 
( 12 , 0.94193 ) 
( 13 , 0.94193 ) 
( 14 , 0.94193 ) 
( 15 , 0.94193 ) 
( 16 , 0.94193 ) 
( 17 , 0.94193 ) 
( 18 , 0.94193 ) 
( 19 , 0.94193 ) 
( 20 , 0.94193 ) 
( 21 , 0.94193 ) 
( 22 , 0.94193 ) 
( 23 , 0.94193 ) 
( 24 , 0.94193 ) 
( 25 , 0.94193 ) 
( 26 , 0.94193 ) 
( 27 , 0.94193 ) 
( 28 , 0.94193 ) 
( 29 , 0.94193 ) 
( 30 , 0.94193 ) 
( 31 , 0.94193 ) 
( 32 , 0.94193 ) 
( 33 , 0.94193 ) 
( 34 , 0.94193 ) 
( 35 , 0.94193 ) 
( 36 , 0.94193 ) 
( 37 , 0.94193 ) 
( 38 , 0.94193 ) 
( 39 , 0.94193 ) 
( 40 , 0.94193 ) 
( 41 , 0.94193 ) 
( 42 , 0.94193 ) 
( 43 , 0.94193 ) 
( 44 , 0.94193 ) 
( 45 , 0.94193 ) 
( 46 , 0.94193 ) 
( 47 , 0.94193 ) 
( 48 , 0.94193 ) 
( 49 , 0.94193 ) 
( 50 , 0.94193 ) 
( 51 , 0.94193 ) 
( 52 , 0.94193 ) 
( 53 , 0.94193 ) 
( 54 , 0.94193 ) 
( 55 , 0.94193 ) 
( 56 , 0.94193 ) 
( 57 , 0.94193 ) 
( 58 , 0.94193 ) 
( 59 , 0.94193 ) 
( 60 , 0.94193 ) 
( 61 , 0.94193 ) 
( 62 , 0.94193 ) 
( 63 , 0.94193 ) 
( 64 , 0.94193 ) 
( 65 , 0.94193 ) 
( 66 , 0.94193 ) 
( 67 , 0.94193 ) 
( 68 , 0.94193 ) 
( 69 , 0.94193 ) 
( 70 , 0.94193 ) 
( 71 , 0.94193 ) 
( 72 , 0.94193 ) 
( 73 , 0.94193 ) 
( 74 , 0.94193 ) 
( 75 , 0.94193 ) 
( 76 , 0.94193 ) 
( 77 , 0.94193 ) 
( 78 , 0.94193 ) 
( 79 , 0.94193 ) 
( 80 , 0.94193 ) 
( 81 , 0.94193 ) 
( 82 , 0.94193 ) 
( 83 , 0.94193 ) 
( 84 , 0.94193 ) 
( 85 , 0.94193 ) 
( 86 , 0.94193 ) 
( 87 , 0.94193 ) 
( 88 , 0.94193 ) 
( 89 , 0.94193 ) 
( 90 , 0.94193 ) 
( 91 , 0.94193 ) 
( 92 , 0.94193 ) 
( 93 , 0.94193 ) 
( 94 , 0.94193 ) 
( 95 , 0.94193 ) 
( 96 , 0.94193 ) 
( 97 , 0.94193 ) 
( 98 , 0.94193 ) 
( 99 , 0.94193 ) 
( 100 , 0.94193 ) 
( 101 , 0.94193 ) 
( 102 , 0.94193 ) 
( 103 , 0.94193 ) 
( 104 , 0.94193 ) 
( 105 , 0.94193 ) 
( 106 , 0.94193 ) 
( 107 , 0.94193 ) 
( 108 , 0.94193 ) 
( 109 , 0.94193 ) 
( 110 , 0.94193 ) 
( 111 , 0.94193 ) 
( 112 , 0.94193 ) 
( 113 , 0.94193 ) 
( 114 , 0.94193 ) 
( 115 , 0.94193 ) 
( 116 , 0.94193 ) 
( 117 , 0.94193 ) 
( 118 , 0.94193 ) 
( 119 , 0.94193 ) 
( 120 , 0.94193 ) 
( 121 , 0.94193 ) 
( 122 , 0.94193 ) 
( 123 , 0.94193 ) 
( 124 , 0.94193 ) 
( 125 , 0.94193 ) 
( 126 , 0.94193 ) 
( 127 , 0.94193 ) 
( 128 , 0.94193 ) 
( 129 , 0.94193 ) 
( 130 , 0.94193 ) 
( 131 , 0.94193 ) 
( 132 , 0.94193 ) 
( 133 , 0.94193 ) 
( 134 , 0.94193 ) 
( 135 , 0.94193 ) 
( 136 , 0.94193 ) 
( 137 , 0.94193 ) 
( 138 , 0.94193 ) 
( 139 , 0.94193 ) 
( 140 , 0.94193 ) 
( 141 , 0.94193 ) 
( 142 , 0.94193 ) 
( 143 , 0.94193 ) 
( 144 , 0.94193 ) 
( 145 , 0.94193 ) 
( 146 , 0.94193 ) 
( 147 , 0.94193 ) 
( 148 , 0.94193 ) 
( 149 , 0.94193 ) 
( 150 , 0.94193 ) 
( 151 , 0.94193 ) 
( 152 , 0.94193 ) 
( 153 , 0.94193 ) 
( 154 , 0.94193 ) 
( 155 , 0.94193 ) 
( 156 , 0.94193 ) 
( 157 , 0.94193 ) 
( 158 , 0.94193 ) 
( 159 , 0.94193 ) 
( 160 , 0.94193 ) 
( 161 , 0.94193 ) 
( 162 , 0.94193 ) 
( 163 , 0.94193 ) 
( 164 , 0.94193 ) 
( 165 , 0.94193 ) 
( 166 , 0.94193 ) 
( 167 , 0.94193 ) 
( 168 , 0.94193 ) 
( 169 , 0.94193 ) 
( 170 , 0.94193 ) 
( 171 , 0.94193 ) 
( 172 , 0.94193 ) 
( 173 , 0.94193 ) 
( 174 , 0.94193 ) 
( 175 , 0.94193 ) 
( 176 , 0.94193 ) 
( 177 , 0.94193 ) 
( 178 , 0.94193 ) 
( 179 , 0.94193 ) 
( 180 , 0.94193 ) 
( 181 , 0.94193 ) 
( 182 , 0.94193 ) 
( 183 , 0.94193 ) 
( 184 , 0.94193 ) 
( 185 , 0.94193 ) 
( 186 , 0.94193 ) 
( 187 , 0.94193 ) 
( 188 , 0.94193 ) 
( 189 , 0.94193 ) 
( 190 , 0.94193 ) 
( 191 , 0.94193 ) 
( 192 , 0.94193 ) 
( 193 , 0.94193 ) 
( 194 , 0.94193 ) 
( 195 , 0.94193 ) 
( 196 , 0.94193 ) 
( 197 , 0.94193 ) 
( 198 , 0.94193 ) 
( 199 , 0.94193 ) 
( 200 , 0.94193 ) 
( 201 , 0.94193 ) 
( 202 , 0.94193 ) 
( 203 , 0.94193 ) 
( 204 , 0.94193 ) 
( 205 , 0.94193 ) 
( 206 , 0.94193 ) 
( 207 , 0.94193 ) 
( 208 , 0.94193 ) 
( 209 , 0.94193 ) 
( 210 , 0.94193 ) 
( 211 , 0.94193 ) 
( 212 , 0.94193 ) 
( 213 , 0.94193 ) 
( 214 , 0.94193 ) 
( 215 , 0.94193 ) 
( 216 , 0.94193 ) 
( 217 , 0.94193 ) 
( 218 , 0.94193 ) 
( 219 , 0.94193 ) 
( 220 , 0.94193 ) 
( 221 , 0.94193 ) 
( 222 , 0.94193 ) 
( 223 , 0.94193 ) 
( 224 , 0.94193 ) 
( 225 , 0.94193 ) 
( 226 , 0.94193 ) 
( 227 , 0.94193 ) 
( 228 , 0.94193 ) 
( 229 , 0.94193 ) 
( 230 , 0.94193 ) 
( 231 , 0.94193 ) 
( 232 , 0.94193 ) 
( 233 , 0.94193 ) 
( 234 , 0.94193 ) 
( 235 , 0.94193 ) 
( 236 , 0.94193 ) 
( 237 , 0.94193 ) 
( 238 , 0.94193 ) 
( 239 , 0.94193 ) 
( 240 , 0.94193 ) 
( 241 , 0.94193 ) 
( 242 , 0.94193 ) 
( 243 , 0.94193 ) 
( 244 , 0.94193 ) 
( 245 , 0.94193 ) 
( 246 , 0.94193 ) 
( 247 , 0.94193 ) 
( 248 , 0.94193 ) 
( 249 , 0.94193 ) 
( 250 , 0.94193 ) 
( 251 , 0.94193 ) 
( 252 , 0.94193 ) 
( 253 , 0.94193 ) 
( 254 , 0.94193 ) 
( 255 , 0.94193 ) 
( 256 , 0.94193 ) 
( 257 , 0.94193 ) 
( 258 , 0.94193 ) 
( 259 , 0.94193 ) 
( 260 , 0.94193 ) 
( 261 , 0.94193 ) 
( 262 , 0.94193 ) 
( 263 , 0.94193 ) 
( 264 , 0.94193 ) 
( 265 , 0.94193 ) 
( 266 , 0.94193 ) 
( 267 , 0.94193 ) 
( 268 , 0.94193 ) 
( 269 , 0.94193 ) 
( 270 , 0.94193 ) 
( 271 , 0.94193 ) 
( 272 , 0.94193 ) 
( 273 , 0.94193 ) 
( 274 , 0.94193 ) 
( 275 , 0.94193 ) 
( 276 , 0.94193 ) 
( 277 , 0.94193 ) 
( 278 , 0.94193 ) 
( 279 , 0.94193 ) 
( 280 , 0.94193 ) 
( 281 , 0.94193 ) 
( 282 , 0.94193 ) 
( 283 , 0.94193 ) 
( 284 , 0.94193 ) 
( 285 , 0.94193 ) 
( 286 , 0.94193 ) 
( 287 , 0.94193 ) 
( 288 , 0.94193 ) 
( 289 , 0.94193 ) 
( 290 , 0.94193 ) 
( 291 , 0.94193 ) 
( 292 , 0.94193 ) 
( 293 , 0.94193 ) 
( 294 , 0.94193 ) 
( 295 , 0.94193 ) 
( 296 , 0.94193 ) 
( 297 , 0.94193 ) 
( 298 , 0.94193 ) 
( 299 , 0.94193 ) 
( 300 , 0.94193 ) 
( 301 , 0.94193 ) 
( 302 , 0.94193 ) 
( 303 , 0.94193 ) 
( 304 , 0.94193 ) 
( 305 , 0.94193 ) 
( 306 , 0.94193 ) 
( 307 , 0.94193 ) 
( 308 , 0.94193 ) 
( 309 , 0.94193 ) 
( 310 , 0.94193 ) 
( 311 , 0.94193 ) 
( 312 , 0.94193 ) 
( 313 , 0.94193 ) 
( 314 , 0.94193 ) 
( 315 , 0.94193 ) 
( 316 , 0.94193 ) 
( 317 , 0.94193 ) 
( 318 , 0.94193 ) 
( 319 , 0.94193 ) 
( 320 , 0.94193 ) 
( 321 , 0.94193 ) 
( 322 , 0.94193 ) 
( 323 , 0.94193 ) 
( 324 , 0.94193 ) 
( 325 , 0.94193 ) 
( 326 , 0.94193 ) 
( 327 , 0.94193 ) 
( 328 , 0.94193 ) 
( 329 , 0.94193 ) 
( 330 , 0.94193 ) 
( 331 , 0.94193 ) 
( 332 , 0.94193 )};
         \addplot[blue] coordinates {
( 1 , 0.32696 ) 
( 2 , 0.74026 ) 
( 3 , 0.75677 ) 
( 4 , 0.92741 ) 
( 5 , 0.82187 ) 
( 6 , 0.95248 ) 
( 7 , 0.85046 ) 
( 8 , 0.95946 ) 
( 9 , 0.86863 ) 
( 10 , 0.96258 ) 
( 11 , 0.88140 ) 
( 12 , 0.96420 ) 
( 13 , 0.89088 ) 
( 14 , 0.96479 ) 
( 15 , 0.89820 ) 
( 16 , 0.96453 ) 
( 17 , 0.90403 ) 
( 18 , 0.96358 ) 
( 19 , 0.90877 ) 
( 20 , 0.96213 ) 
( 21 , 0.91267 ) 
( 22 , 0.96036 ) 
( 23 , 0.91593 ) 
( 24 , 0.95842 ) 
( 25 , 0.91868 ) 
( 26 , 0.95644 ) 
( 27 , 0.92101 ) 
( 28 , 0.95449 ) 
( 29 , 0.92300 ) 
( 30 , 0.95263 ) 
( 31 , 0.92472 ) 
( 32 , 0.95089 ) 
( 33 , 0.92620 ) 
( 34 , 0.94928 ) 
( 35 , 0.92748 ) 
( 36 , 0.94783 ) 
( 37 , 0.92860 ) 
( 38 , 0.94651 ) 
( 39 , 0.92957 ) 
( 40 , 0.94533 ) 
( 41 , 0.93042 ) 
( 42 , 0.94428 ) 
( 43 , 0.93117 ) 
( 44 , 0.94335 ) 
( 45 , 0.93183 ) 
( 46 , 0.94252 ) 
( 47 , 0.93240 ) 
( 48 , 0.94179 ) 
( 49 , 0.93291 ) 
( 50 , 0.94115 ) 
( 51 , 0.93335 ) 
( 52 , 0.94059 ) 
( 53 , 0.93374 ) 
( 54 , 0.94009 ) 
( 55 , 0.93409 ) 
( 56 , 0.93966 ) 
( 57 , 0.93439 ) 
( 58 , 0.93928 ) 
( 59 , 0.93466 ) 
( 60 , 0.93895 ) 
( 61 , 0.93490 ) 
( 62 , 0.93866 ) 
( 63 , 0.93511 ) 
( 64 , 0.93841 ) 
( 65 , 0.93529 ) 
( 66 , 0.93818 ) 
( 67 , 0.93546 ) 
( 68 , 0.93799 ) 
( 69 , 0.93560 ) 
( 70 , 0.93782 ) 
( 71 , 0.93573 ) 
( 72 , 0.93768 ) 
( 73 , 0.93584 ) 
( 74 , 0.93755 ) 
( 75 , 0.93594 ) 
( 76 , 0.93744 ) 
( 77 , 0.93603 ) 
( 78 , 0.93734 ) 
( 79 , 0.93611 ) 
( 80 , 0.93726 ) 
( 81 , 0.93618 ) 
( 82 , 0.93719 ) 
( 83 , 0.93624 ) 
( 84 , 0.93713 ) 
( 85 , 0.93630 ) 
( 86 , 0.93707 ) 
( 87 , 0.93635 ) 
( 88 , 0.93703 ) 
( 89 , 0.93639 ) 
( 90 , 0.93699 ) 
( 91 , 0.93643 ) 
( 92 , 0.93695 ) 
( 93 , 0.93647 ) 
( 94 , 0.93693 ) 
( 95 , 0.93650 ) 
( 96 , 0.93690 ) 
( 97 , 0.93653 ) 
( 98 , 0.93688 ) 
( 99 , 0.93655 ) 
( 100 , 0.93686 ) 
( 101 , 0.93658 ) 
( 102 , 0.93685 ) 
( 103 , 0.93660 ) 
( 104 , 0.93684 ) 
( 105 , 0.93662 ) 
( 106 , 0.93683 ) 
( 107 , 0.93664 ) 
( 108 , 0.93682 ) 
( 109 , 0.93666 ) 
( 110 , 0.93682 ) 
( 111 , 0.93667 ) 
( 112 , 0.93682 ) 
( 113 , 0.93669 ) 
( 114 , 0.93682 ) 
( 115 , 0.93670 ) 
( 116 , 0.93682 ) 
( 117 , 0.93672 ) 
( 118 , 0.93682 ) 
( 119 , 0.93673 ) 
( 120 , 0.93682 ) 
( 121 , 0.93674 ) 
( 122 , 0.93682 ) 
( 123 , 0.93676 ) 
( 124 , 0.93682 ) 
( 125 , 0.93677 ) 
( 126 , 0.93683 ) 
( 127 , 0.93678 ) 
( 128 , 0.93683 ) 
( 129 , 0.93679 ) 
( 130 , 0.93684 ) 
( 131 , 0.93680 ) 
( 132 , 0.93685 ) 
( 133 , 0.93681 ) 
( 134 , 0.93685 ) 
( 135 , 0.93683 ) 
( 136 , 0.93686 ) 
( 137 , 0.93684 ) 
( 138 , 0.93687 ) 
( 139 , 0.93685 ) 
( 140 , 0.93687 ) 
( 141 , 0.93686 ) 
( 142 , 0.93688 ) 
( 143 , 0.93687 ) 
( 144 , 0.93689 ) 
( 145 , 0.93688 ) 
( 146 , 0.93690 ) 
( 147 , 0.93689 ) 
( 148 , 0.93691 ) 
( 149 , 0.93690 ) 
( 150 , 0.93692 ) 
( 151 , 0.93691 ) 
( 152 , 0.93693 ) 
( 153 , 0.93693 ) 
( 154 , 0.93694 ) 
( 155 , 0.93694 ) 
( 156 , 0.93695 ) 
( 157 , 0.93695 ) 
( 158 , 0.93696 ) 
( 159 , 0.93696 ) 
( 160 , 0.93697 ) 
( 161 , 0.93697 ) 
( 162 , 0.93699 ) 
( 163 , 0.93699 ) 
( 164 , 0.93700 ) 
( 165 , 0.93700 ) 
( 166 , 0.93701 ) 
( 167 , 0.93701 ) 
( 168 , 0.93702 ) 
( 169 , 0.93702 ) 
( 170 , 0.93703 ) 
( 171 , 0.93704 ) 
( 172 , 0.93705 ) 
( 173 , 0.93705 ) 
( 174 , 0.93706 ) 
( 175 , 0.93706 ) 
( 176 , 0.93707 ) 
( 177 , 0.93708 ) 
( 178 , 0.93709 ) 
( 179 , 0.93709 ) 
( 180 , 0.93710 ) 
( 181 , 0.93711 ) 
( 182 , 0.93712 ) 
( 183 , 0.93712 ) 
( 184 , 0.93713 ) 
( 185 , 0.93714 ) 
( 186 , 0.93714 ) 
( 187 , 0.93715 ) 
( 188 , 0.93716 ) 
( 189 , 0.93717 ) 
( 190 , 0.93717 ) 
( 191 , 0.93718 ) 
( 192 , 0.93719 ) 
( 193 , 0.93720 ) 
( 194 , 0.93721 ) 
( 195 , 0.93721 ) 
( 196 , 0.93722 ) 
( 197 , 0.93723 ) 
( 198 , 0.93724 ) 
( 199 , 0.93725 ) 
( 200 , 0.93725 ) 
( 201 , 0.93726 ) 
( 202 , 0.93727 ) 
( 203 , 0.93728 ) 
( 204 , 0.93729 ) 
( 205 , 0.93730 ) 
( 206 , 0.93731 ) 
( 207 , 0.93731 ) 
( 208 , 0.93732 ) 
( 209 , 0.93733 ) 
( 210 , 0.93734 ) 
( 211 , 0.93735 ) 
( 212 , 0.93736 ) 
( 213 , 0.93737 ) 
( 214 , 0.93738 ) 
( 215 , 0.93739 ) 
( 216 , 0.93740 ) 
( 217 , 0.93740 ) 
( 218 , 0.93742 ) 
( 219 , 0.93742 ) 
( 220 , 0.93744 ) 
( 221 , 0.93744 ) 
( 222 , 0.93746 ) 
( 223 , 0.93746 ) 
( 224 , 0.93748 ) 
( 225 , 0.93748 ) 
( 226 , 0.93750 ) 
( 227 , 0.93750 ) 
( 228 , 0.93752 ) 
( 229 , 0.93752 ) 
( 230 , 0.93754 ) 
( 231 , 0.93754 ) 
( 232 , 0.93756 ) 
( 233 , 0.93756 ) 
( 234 , 0.93758 ) 
( 235 , 0.93759 ) 
( 236 , 0.93761 ) 
( 237 , 0.93761 ) 
( 238 , 0.93763 ) 
( 239 , 0.93763 ) 
( 240 , 0.93765 ) 
( 241 , 0.93765 ) 
( 242 , 0.93768 ) 
( 243 , 0.93767 ) 
( 244 , 0.93770 ) 
( 245 , 0.93769 ) 
( 246 , 0.93773 ) 
( 247 , 0.93771 ) 
( 248 , 0.93775 ) 
( 249 , 0.93774 ) 
( 250 , 0.93778 ) 
( 251 , 0.93776 ) 
( 252 , 0.93780 ) 
( 253 , 0.93778 ) 
( 254 , 0.93783 ) 
( 255 , 0.93780 ) 
( 256 , 0.93786 ) 
( 257 , 0.93782 ) 
( 258 , 0.93788 ) 
( 259 , 0.93785 ) 
( 260 , 0.93791 ) 
( 261 , 0.93787 ) 
( 262 , 0.93794 ) 
( 263 , 0.93789 ) 
( 264 , 0.93797 ) 
( 265 , 0.93791 ) 
( 266 , 0.93800 ) 
( 267 , 0.93793 ) 
( 268 , 0.93803 ) 
( 269 , 0.93795 ) 
( 270 , 0.93807 ) 
( 271 , 0.93797 ) 
( 272 , 0.93810 ) 
( 273 , 0.93799 ) 
( 274 , 0.93814 ) 
( 275 , 0.93801 ) 
( 276 , 0.93817 ) 
( 277 , 0.93803 ) 
( 278 , 0.93821 ) 
( 279 , 0.93804 ) 
( 280 , 0.93825 ) 
( 281 , 0.93806 ) 
( 282 , 0.93829 ) 
( 283 , 0.93807 ) 
( 284 , 0.93834 ) 
( 285 , 0.93808 ) 
( 286 , 0.93839 ) 
( 287 , 0.93809 ) 
( 288 , 0.93843 ) 
( 289 , 0.93810 ) 
( 290 , 0.93849 ) 
( 291 , 0.93810 ) 
( 292 , 0.93854 ) 
( 293 , 0.93810 ) 
( 294 , 0.93860 ) 
( 295 , 0.93810 ) 
( 296 , 0.93867 ) 
( 297 , 0.93809 ) 
( 298 , 0.93874 ) 
( 299 , 0.93808 ) 
( 300 , 0.93881 ) 
( 301 , 0.93806 ) 
( 302 , 0.93889 ) 
( 303 , 0.93804 ) 
( 304 , 0.93898 ) 
( 305 , 0.93801 ) 
( 306 , 0.93908 ) 
( 307 , 0.93797 ) 
( 308 , 0.93918 ) 
( 309 , 0.93792 ) 
( 310 , 0.93930 ) 
( 311 , 0.93786 ) 
( 312 , 0.93943 ) 
( 313 , 0.93779 ) 
( 314 , 0.93957 ) 
( 315 , 0.93770 ) 
( 316 , 0.93972 ) 
( 317 , 0.93760 ) 
( 318 , 0.93990 ) 
( 319 , 0.93748 ) 
( 320 , 0.94009 ) 
( 321 , 0.93734 ) 
( 322 , 0.94031 ) 
( 323 , 0.93717 ) 
( 324 , 0.94055 ) 
( 325 , 0.93699 ) 
( 326 , 0.94082 ) 
( 327 , 0.93677 ) 
( 328 , 0.94112 ) 
( 329 , 0.93651 ) 
( 330 , 0.94146 ) 
( 331 , 0.93622 ) 
( 332 , 0.94183 ) };
        \legend{%
$\tfrac{\lambda_1}{\lambda_2} = 0.9991$,
$|\mu_1| = 0.9419$,
$r(n)$,
}
    \end{axis}
\end{tikzpicture}
\caption{{\it Contraction rates for model problem 2. The numerical contraction rates $r(n)$ at iteration $n$ are computed as in \eqref{contraction-rate-H1}. The upper bounds for the rates are given by $\tfrac{\lambda_1}{\lambda_2}$ according to Theorem \ref{thm_conv_basic_inverse_iteration} and by $|\mu_1|$ according to Remark \ref{remark-sharper-rates}.}}
\label{rates_modelproblem2}
\end{figure}

Again, the rate $\frac{\lambda_1}{\lambda_2}$ yields an overshoot and is not sharp for large values of $\beta$, whereas $|\mu_1|$ describes the asymptotic convergence well. In the last iteration, we obtain the contraction rate $r(n)=0.94183$, which is again remarkably close to $|\mu_1|= 0.94192$. Note however that the values for $r(n)$ in the final iterations show an oscillatory behavior, where the values for $r(n)$ are roughly between $0.937$ and $0.9418$. There could be various reasons for this. 
The most simple explanation are rounding errors in our computations which could lead to a tiny offset of $r(n)$. This is not surprising when considering that the iterates $u^n$ are getting close to the numerically computed reference solution in the regime where the oscillations are observed.

However, it is also possible that the asymptotic oscillations are indeed not an artifact, but that there is a mathematical reason for it. This is related to the observation that the differential operator $ v\mapsto \mathcal{G}_{u} ( \hspace{1pt} (1 - 2 \tfrac{\beta}{\lambda} \, |u|^2) v \hspace{1pt} ) $, that describes the weighted eigenvalue problem \eqref{weighted-evp}, is not self-adjoint (neither with respect to the $L^2$- nor the $H^1$-inner product), hence the spectral radius is not an induced operator norm. To understand why this is relevant, we need to have a look at the proof of the Ostrowski theorem (cf. \cite{AHP21NumMath,Shi81}). Here it is exploited that for any $\eps>0$ there exists a norm $\| \cdot \|_{\eps}$ (that is equivalent to the $H^1$-norm) such that $|||  \phi^{\prime}(u) |||_{\eps}:= \sup\limits_{ \| v \|_{\eps}=1} \langle \phi^{\prime}(u), v \rangle \le |\mu_1| + \eps$. Consequently, in a sufficiently small $S_{\eps}$ neighborhood of $u$ it holds for all iterates $u^n$ that $\| u - u^n \|_{\eps} = \| \phi(u) - \phi(u^{n-1}) \|_{\eps} \le (|\mu_1| + 2\eps ) \| u - u^{n-1} \|_{\eps} \le ... \le (|\mu_1| + 2\eps )^n \| u - u^{0} \|_{\eps}$. In other words, the contraction rate $r_{\eps}(n)\approx |\mu_1|$ {\it per iteration} is only guaranteed in the (unknown) $\| \cdot \|_{\eps}$-norm, whereas the error in the $H^1$-norm is also influenced by the unknown norm equivalence constants. This does not matter asymptotically, because (by exploiting the norm equivalence for $\| u - u^n \|_{\eps}$ and $\| u - u^{0} \|_{\eps}$) we still obtain $\| u - u^n \|_{H^1(\Omega)} \le C_{\eps} (|\mu_1| + 2\eps )^n \| u - u^{0} \|_{H^1(\Omega)}$, where the influence of $C_{\eps}$ is vanishing for $n\rightarrow \infty$ and we have an error of order $ (|\mu_1| + 2\eps )^n$ after $n$ iterations. However, this is different when looking at two successive iterations, for which we only obtain the estimate $\| u - u^n \|_{H^1(\Omega)} \le C_{\eps} (|\mu_1| + 2\eps ) \| u - u^{n-1} \|_{H^1(\Omega)}$, which does formally not even guarantee a reduction of the error from one iterate to the next if $C_{\eps} (|\mu_1| + 2\eps )$ is too large. Hence, $C_{\eps}$ can potentially have an influence on how the error changes in individual iterations. In summary, the numerically observed rates $r(n)$ can be slightly polluted by norm equivalence constants. However, over several iterations the influence of these constants has to cancel out, hence leading to a slightly oscillatory behaviour of $r(n)$ in the asymptotic phase, as observed in our experiment.

In conclusion we can say however that our numerically computed value for $|\mu_1|$ yields an accurate upper bound for the observed contraction rates, confirming our main results regarding the estimated convergence speeds for the generalized inverse iteration \eqref{basic-inverse-iterations}.

\def\cprime{$'$}

\end{document}